\documentclass[10pt]{amsart}

\usepackage{latexsym,exscale,enumerate,amsfonts,amssymb,xparse, mathtools}
\usepackage[normalem]{ulem}
\usepackage{amsmath,amsthm,amsfonts,amssymb,amscd, stmaryrd,textcomp}
\usepackage[bookmarks=false]{hyperref}
\usepackage{bbold}
\usepackage[usenames,dvipsnames]{xcolor}
\usepackage{enumitem}
\usepackage{verbatim}
\usepackage{ytableau}

\usepackage{xcolor}
\definecolor{myorange}{rgb}{1,0.647,0}

\definecolor{mypurple}{cmyk}{.6,.9,0, .11}

\numberwithin{equation}{section}

\newcommand{\arxiv}[1]{\href{https://arxiv.org/abs/#1}{\small  arXiv:#1}}

\theoremstyle{definition}
\newtheorem{thm}{Theorem}[section]
\newtheorem{cor}[thm]{Corollary}

\newtheorem{lem}[thm]{Lemma}
\newtheorem{rem}[thm]{Remark}

\newtheorem{prop}[thm]{Proposition}
\newtheorem{defn}[thm]{Definition}
\newtheorem{example}[thm]{Example}

\newtheorem*{ack}{Acknowledgements}

\usepackage{tikz}
\usetikzlibrary{cd}
\usetikzlibrary{snakes}
\usetikzlibrary{decorations.markings}
\usetikzlibrary{decorations.pathreplacing}
\usetikzlibrary{arrows,shapes,positioning}
\usetikzlibrary{decorations.pathmorphing}
\tikzset{anchorbase/.style={baseline={([yshift=-0.5ex]current bounding box.center)}},
tinynodes/.style={font=\tiny,text height=0.75ex,text depth=0.15ex},
smallnodes/.style={font=\scriptsize,text height=0.75ex,text depth=0.15ex},
>={Latex[length=1mm, width=1.5mm]},
overcross/.style={line width=4pt,color=white},
}
\tikzstyle directed=[postaction={decorate,decoration={markings,
    mark=at position #1 with {\arrow{>}}}}]
\tikzstyle rdirected=[postaction={decorate,decoration={markings,
    mark=at position #1 with {\arrow{<}}}}]

\newcommand{\Hom}{{\rm Hom}}

\newcommand{\End}{{\rm End}}

\renewcommand{\to}{\rightarrow}

\def\C{{\mathbb C}}

\def\R{{\mathbb R}}

\newcommand{\gln}[1][n]{\mathfrak{gl}_{#1}}
\newcommand{\spn}[1][2n]{\mathfrak{sp}_{#1}}

\newcommand{\Web}{\mathbf{Web}}

\newcommand{\CS}{\mathcal{C}}
\newcommand{\SG}{\mathfrak{S}}

\newcommand{\ic}{\mathsf{ic}}
\newcommand{\oc}{\mathsf{oc}}
\newcommand{\UD}{\mathrm{OT}}

\DeclareMathOperator{\RS}{RS}
\DeclareMathOperator{\BSW}{BSW}
\DeclareMathOperator{\SW}{SW}
\DeclareMathOperator{\Brauer}{Brauer}

\begin{document}
%

\title[]{Decreasing subsequences and Viennot for oscillating tableaux}

\author{Elijah Bodish}
\address{Department of Mathematics, Indiana University Bloomington, 
Rawles Hall, Bloomington, IN 47405-7106, USA}
\email{ebodish@iu.edu}

\author{Ben Elias}
\address{Department of Mathematics, University of Oregon,
Fenton Hall, Eugene, OR 97403-1222, USA}
\email{belias@uoregon.edu}

\author{David E. V. Rose}
\address{Department of Mathematics, University of North Carolina, 
Phillips Hall CB \#3250, UNC-CH, Chapel Hill, NC 27599-3250, USA}
\email{davidrose@unc.edu, ltatham@live.unc.edu}

\author{Logan Tatham}

\begin{abstract}
We establish an extension of Viennot's geometric (shadow line) construction
to the setting of oscillating tableaux.
We then use this to give a new proof of the Type $C$ analogue
of Schensted's theorem on longest decreasing subsequences.
This pairs with our results from \cite{BERT} on Type $C$ webs to give a
direct proof of a result of Sundaram and Stanley: 
that the dimension of the space of invariant vectors in a 
$2k$-fold tensor product of the vector representation 
of $\spn$ equals the number of $(n+1)$-avoiding matchings of $2k$ points.
\end{abstract}

\maketitle

%
\section{Introduction}
%

The Robinson--Schensted correspondence is a celebrated combinatorial bijection 
between permutations and pairs of standard Young tableaux of the same shape: 
\[ \SG_k
\stackrel{\RS}{\longleftrightarrow} 
\{(P,Q,\lambda) \mid \lambda \vdash k \, , \ P, Q \text{ are standard Young tableaux of shape } \lambda \} \, . 
\] 
Mechanically, one takes a permutation $w \in \SG_k$ written in one-line notation $w(1) w(2) \cdots w(k)$ 
and uses Schensted's \emph{bumping algorithm} 
to add one number $w(t)$ at a time to the \emph{insertion tableau} $P$, 
at each step giving a tableau $P_t$ (of non-standard content) whose shape is a partition $\lambda_t \vdash t$. 
This sequence of partitions is encoded in the \emph{recording tableau} $Q$. 
In other words, the bumping algorithm is a rule to produce from $w$ a sequence of steadily growing tableaux 
$(P_t, \lambda_t)$ for $1 \le t \le k$.

Many useful properties of the permutation $w \in \SG_k$ are encoded within the corresponding triple 
$\RS(w) = (P,Q,\lambda)$, but many other features are obfuscated. 
While one can use Schensted's reverse bumping algorithm to determine $w$ from $\RS(w)$, 
it is not easy to gain intuition for how $w$ behaves simply by looking at $P$ and $Q$.

An example where a feature of a permutation is both well-encoded and obscured by $\RS$ 
is the study of longest increasing (resp. longest decreasing) subsequences. 
An \emph{increasing subsequence} is a sequence $1 \le t_1 < t_2 < \cdots < t_\ell \le k$
such that $w(t_1) < w(t_2) < \cdots < w(t_\ell)$; the length of the subsequence is $\ell$. 
Decreasing subsequences are defined analogously. The following is a classic result due to Schensted.

\begin{thm}[{\cite{Schensted}}]\label{thm:longest}
Let $w \in \SG_k$. The length of the longest increasing (resp. decreasing) subsequence of $w$ is equal to 
the number of columns\footnote{Here we use synecdoche: if $\RS(w) = (P,Q,\lambda)$ then 
by ``the number of columns in $\RS(w)$'' we mean the number of columns in $\lambda$. 
We continue to use this convention throughout the paper.} (resp. rows) in $\RS(w)$.
\end{thm}

Despite knowing their length, it is not obvious how to read the longest increasing and decreasing 
subsequences from the triple $(P,Q,\lambda)$.

\begin{example}\label{ex:PQ1}
Consider the permutation $w \in \SG_{19}$ written in one-line notation as
\[
w=2 \ 9 \ 1 \ 15 \ 4 \ 7 \ 13 \ 18 \ 11 \ 19 \ 5 \ 14 \ 3 \ 10 \ 6 \ 17 \ 8 \ 16 \ 12 \, .
\]
We have that
\[
P=
{\footnotesize
\begin{ytableau}
1 & 3 & 5 & 6 & 8 & 12 \cr
2 & 4 & 10 & 14 & 16 \cr
7 & 11 & 17 & 19 \cr
9 & 13 & 18 \cr
15 \cr
\end{ytableau}
}
\, , \quad
Q=
{\footnotesize
\begin{ytableau}
1 & 2 & 4 & 7 & 8 & 10 \cr
3 & 5 & 6 & 12 & 16 \cr
9 & 14 & 17 & 18 \cr
11 & 15 & 19 \cr
13 \cr
\end{ytableau}
}
\]
Correspondingly, the longest increasing subsequence has length $6$, and one such subsequence 
is (indicated in \textbf{bold}):
\[
w= \mathbf{2} \ 9 \ 1 \ 15 \ \mathbf{4} \ 7 \ 13 \ 18 \ 11 \ 19 \ \mathbf{5} \ 14 \ 3 \ 
	10 \ \mathbf{6} \ 17 \ \mathbf{8} \ 16 \ \mathbf{12}
\]
(there are many others). It is not clear how this subsequence is encoded in $P$ and $Q$.
On the other hand, 
the longest decreasing subsequence has length $5$, and one such subsequence is
\[
w=2 \ 9 \ 1 \ \mathbf{15} \ 4 \ 7 \ \mathbf{13} \ 18 \ \mathbf{11} \ 19 \ 5 \ 14 \ 3 \ 
	\textbf{10} \ 6 \ 17 \ \textbf{8} \ 16 \ 12 \, .
\]
In this case, it is tempting to believe that we obtain a longest decreasing subsequence 
by letting $w(t_1)$ be the entry in the last row of $P$ and then successively taking the largest 
entry that is smaller in the previous row. However, this is \textbf{false}, as our next example shows.
\end{example}

\begin{example}\label{ex:PQ2}
Consider the permutation $w \in \SG_{19}$ written in one-line notation as
\[
w = 19 \ 10 \ 1 \ 8 \ 18 \ 12 \ 13 \ 11 \ 16 \ 14 \ 3 \ 9 \ 7 \ 4 \ 6 \ 2 \ 15 \ 17 \ 5 \, .
\]
We have that
\[
P=
{\footnotesize
\begin{ytableau}
1 & 2 & 4 & 5 & 14 & 15 & 17 \cr
3 & 6 & 13 \cr
7 & 9 & 16 \cr
8 & 11 \cr
10 \cr
12 \cr
18 \cr
19 \cr
\end{ytableau}
}
\, , \quad
Q=
{\footnotesize
\begin{ytableau}
1 & 4 & 5 & 7 & 9 & 17 & 18 \cr
2 & 6 & 10 \cr
3 & 12 & 15 \cr
8 & 19 \cr
11 \cr
13 \cr
14 \cr
16 \cr
\end{ytableau}
}
\]
The longest increasing subsequence has length $7$, 
and one such subsequence is 
\[
w = 19 \ 10 \ \mathbf{1} \ \mathbf{8} \ 18 \ \mathbf{12} \ \mathbf{13} \ 11 \ 16 \ \mathbf{14} \ 
	3 \ 9 \ 7 \ 4 \ 6 \ 2 \ \mathbf{15} \ \mathbf{17} \ 5 \, .
\]
The longest decreasing subsequence has length $8$, 
and one such subsequence is 
\[
w = \mathbf{19} \ 10 \ 1 \ 8 \ \mathbf{18} \ 12 \ 13 \ 11 \ \mathbf{16} \ \mathbf{14} \ 
	3 \ \mathbf{9} \ \mathbf{7} \ \mathbf{4} \ 6 \ \mathbf{2} \ 15 \ 17 \ 5 \, .
\]
In both cases, 
it is not clear how to read these subsequences off from $P$ or $Q$. 
In particular, the schema for finding a longest decreasing subsequence in Example \ref{ex:PQ1} 
\textbf{fails} to produce such a sequence here.
Indeed, no longest decreasing sequence contains the number $10$, the lone entry in its row in $P$.
\end{example}

\begin{rem}[Relation to $\gln$ representation theory]
Theorem \ref{thm:longest} links longest decreasing sequences to the representation theory of the Lie algebra $\gln$. 
Consider the vector representation $V = \C^n$ of $\gln$. 
Schur--Weyl duality between $\SG_k$ and $\gln$ shows that there is a surjective $\C$-algebra homomorphism
\[
\SW \colon 
\C[\SG_k] \to \End_{\gln}\big( V^{\otimes k} \big) \, .
\]
This homomorphism is injective if and only if $k \le n$. 
The group algebra $\C[\SG_k]$ has a basis $\{b_w\}$ indexed by permutations $w \in \SG_k$
(the $q=1$ specialization of the Kazhdan--Lusztig basis) such that $\ker(\SW)$ is 
spanned\footnote{We are unaware of the best reference for this well-known fact. 
As a representation of $\C[\SG_k]$, $V^{\otimes k}$ is a sum of Specht modules 
for partitions $\lambda$ with at most $n$ rows. 
In \cite[Theorem 6.5.3]{BjBr}, one can find a proof of the theorem 
(claimed earlier by Kazhdan and Lusztig) 
that the cell modules for the Hecke algebra agree with the Specht modules with the expected labeling. 
Thus, Kazhdan-Lusztig basis elements in lower cells than $\lambda$ 
will kill the Specht module associated to $\lambda$.} 
by 
\[
\{b_w \mid \RS(w) \text{ has at least $n+1$ rows} \} \, . 
\]
Consequently, 
the dimension of the space of $\gln$-invariant vectors in $V^{\otimes k} \otimes \big(V^\vee)^{\otimes k}$
is equal to
the number of permutations $w \in \SG_k$ such that $\RS(w)$ has $n$ or fewer rows. 
By Theorem \ref{thm:longest}, this is the same as the number of permutations 
that avoid a length $n+1$ decreasing sequence.

\noindent \textbf{Note:} A reader lacking background on (or disinterested in) such 
aspects of Lie algebra representation theory can safely read on. 
We have relegated the representation-theoretic implications of the combinatorics discussed in this paper to 
this Remark and Remark \ref{rem:spn}.
\end{rem}

In the present paper, motivated by representation theory, 
we are particularly interested in the case of longest \emph{decreasing} subsequences. 
While Schensted's proof of Theorem \ref{thm:longest} in the case of increasing subsequences is straightforward, 
his arguments for decreasing subsequences are indirect.
Work of Viennot \cite{Viennot} remedies this, providing
a direct and easy proof of Theorem \ref{thm:longest} for decreasing subsequences. 
In fact, our main interest in this paper is not the Robinson--Schensted correspondence, 
but rather a type $C$ generalization thereof that is
related to the representation theory of the Lie algebra $\spn$. 
As we show in \S \ref{sec:updown}, 
Viennot's proof of Theorem \ref{thm:longest} extends to type $C$.
However, we first finish discussing the story in type $A$.

Viennot's work proceeds via a graphical interpretation of the Robinson--Schensted correspondence 
in terms of certain colored segments in the first quadrant of $\R^2$ that we call \emph{Viennot diagrams}.
We recall this construction and give several examples in \S \ref{sec:Viennot}.
Viennot's methods are most well-known because they give an easy proof of the following property 
(which is hard to prove using the bumping algorithm): 
if $\RS(w) = (P,Q,\lambda)$, then $\RS(w^{-1}) = (Q,P,\lambda)$. 
Viennot's simple proof of this fact follows by reflecting a Viennot diagram 
along the diagonal $y = x$ inside $\R^2$, a symmetry which swaps $w$ and $w^{-1}$, as well as $P$ and $Q$. 
However, Viennot also provides a straightforward proof of Theorem \ref{thm:longest}.
To keep the discussion self-contained 
(and since the original French text \cite{Viennot} is sufficiently 
difficult\footnote{Indeed, difficult enough that the authors were able to independently 
prove the results in \S \ref{sec:Viennot} on decreasing/increasing subsequences 
(and our generalization to oscillating tableaux in \S \ref{sec:updown}) 
before realizing we had been scooped (on the former) by none other than Viennot himself!}
to track down), 
we recall Viennot's argument for longest
decreasing subsequences in Proposition \ref{prop:corners} and Corollary \ref{cor:decreasing}
and for longest increasing subsequences in Proposition \ref{prop:increasing}.
This provides an intuitive algorithm to find longest decreasing/increasing 
subsequences; see the discussion around Lemma \ref{lem:algorithm}. 
Another useful feature of Viennot's construction is that it records a ``timeline'' of the bumping algorithm; 
see Definition \ref{def:Viennot} and Example \ref{ex:timeline}. 
This will be a crucial feature in our extension of this construction to type $C$ combinatorics.

\subsection{The Sundaram--Stanley bijection} \label{sec:introtypeC}

The type $C$ analogue of the Robinson--Schensted correspondence 
(dealing with the combinatorics arising in Brauer--Schur--Weyl duality)
is due to Sundaram \cite{SundaramThesis}, with a variant given by Stanley \cite{Stanley}. 
In this correspondence, permutations in $\SG_k$ are generalized by matchings of $2k$ points, 
and pairs of standard Young tableaux are replaced by so-called \emph{oscillating tableau}. 
We now describe this bijection, which is easy to motivate using \emph{string diagrams} for permutations in $\SG_k$.

Suppose we are given a string diagram $\mathcal{D}$ for a permutation $w(\mathcal{D}) \in \SG_k$, e.g.
\begin{equation}\label{eq:SD}
\mathcal{D} = 
\begin{tikzpicture}[anchorbase, scale=.5,smallnodes]
\draw[very thick] (1,0) to [out=90,in=270] (2,3) node[above=-2pt]{$2$};
\draw[very thick] (2,0) to [out=90,in=270] (4,3) node[above=-2pt]{$4$};
\draw[very thick] (3,0)  to [out=90,in=270] (4,1.5) to [out=90,in=270] (3,3) node[above=-2pt]{$3$};
\draw[very thick] (4,0) to [out=90,in=270] (1,3) node[above=-2pt]{$1$};
\end{tikzpicture}
\implies
w(\mathcal{D}) = 2 \ 4 \ 3 \ 1 \, .
\end{equation}
We read string diagrams from bottom-to-top, 
thus the one-line notation for $w(\mathcal{D})$ can be obtained by labeling the points at the 
top of the strands $1,\ldots,k$ from left-to-right as in \eqref{eq:SD}, and ``sliding'' the numbers  
along the strands to the bottom.
By ``bending'' the top points around to the right, this permutation determines a matching of $2k$ fixed points on a line 
wherein none of first (or last) $k$ points are matched with each other, e.g.
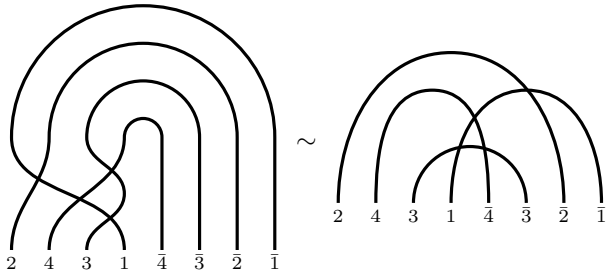
\[
\begin{tikzpicture}[anchorbase, scale=.5,smallnodes]
\draw[very thick] (1,0) node[below]{$2$} to [out=90,in=270] (2,3) 
	to [out=90,in=180] (4.5,5.5) to [out=0,in=90] (7,3) to (7,0) node[below]{$\bar{2}$};
\draw[very thick] (2,0) node[below]{$4$} to [out=90,in=270] (4,3) 
	to [out=90,in=180] (4.5, 3.5) to [out=0,in=90] (5,3) to (5,0) node[below]{$\bar{4}$};
\draw[very thick] (3,0) node[below]{$3$} to [out=90,in=270] (4,1.5) to [out=90,in=270] (3,3) 
	to [out=90,in=180] (4.5,4.5) to [out=0,in=90] (6,3) to (6,0) node[below]{$\bar{3}$};
\draw[very thick] (4,0) node[below]{$1$} to [out=90,in=270] (1,3) 
	to [out=90,in=180] (4.5,6.5) to [out=0,in=90] (8,3) to (8,0) node[below]{$\bar{1}$};
\end{tikzpicture}
\sim
\begin{tikzpicture}[anchorbase, scale=.5,smallnodes]
\draw[very thick] (1,0) node[below]{$2$} to [out=90,in=180] (4,4) to [out=0,in=90] (7,0) node[below]{$\bar{2}$};
\draw[very thick] (2,0) node[below]{$4$} to [out=90,in=180] (3.5,3) to [out=0,in=90]  (5,0) node[below]{$\bar{4}$};
\draw[very thick] (3,0) node[below]{$3$} to [out=90,in=180] (4.5,1.5) to [out=0,in=90] (6,0) node[below]{$\bar{3}$};
\draw[very thick] (4,0) node[below]{$1$} to [out=90,in=180] (6,3) to [out=0,in=90](8,0) node[below]{$\bar{1}$};
\end{tikzpicture}
\]
Here, we label the right endpoints of the strands in the matching by the labels $\bar{1},\ldots,\bar{k}$ 
from \emph{right-to-left}, 
and label the left endpoints with the corresponding un-barred symbol. 
The one-line notation for this matching is therefore 
\[ 
2\ 4\ 3\ 1\ \bar{4} \ \bar{3} \ \bar{2} \ \bar{1} \, . 
\]
It is clearly determined by the one-line notation $2431$ for $w$, 
since the last $k$ symbols are simply $\bar{k}, \ldots, \bar{1}$.

Analogously, given an arbitrary matching\footnote{We will always assume that the corresponding 
string diagram is \emph{reduced}, meaning that pairs of strands cross each other the 
minimal possible number of times.} 
of $2k$ points on a line, 
we can label the right endpoints of the strands $\bar{k}, \ldots, \bar{1}$ 
and label the left endpoints accordingly.
This associates to each such matching $\mathcal{M}$ 
a word $w(\mathcal{M})$ using each of the symbols $\{1,\ldots,k,\bar{1},\ldots,\bar{k}\}$ once. 
In this word, the symbol $i$ must appear before the symbol $\bar{i}$, and the symbols 
$\bar{k},\ldots,\bar{1}$ appear in decreasing order. For example:
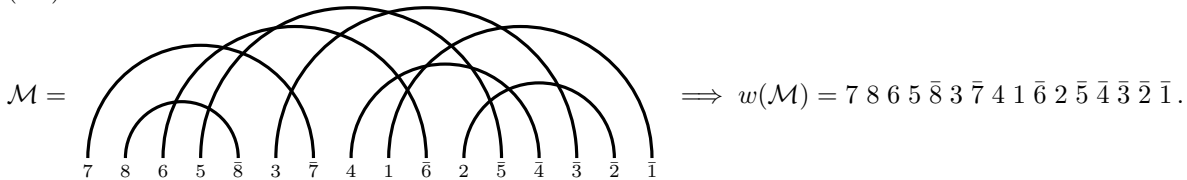
\begin{equation}\label{eq:matchingex}
\mathcal{M} = 
\begin{tikzpicture}[anchorbase, scale=.5,smallnodes]
\draw[very thick] (1,0) node[below]{$7$} to [out=90,in=180] (4,3) to [out=0,in=90] (7,0) node[below]{$\bar{7}$};
\draw[very thick] (2,0) node[below]{$8$} to [out=90,in=180] (3.5,1.5) to [out=0,in=90]  (5,0) node[below]{$\bar{8}$};
\draw[very thick] (3,0) node[below]{$6$} to [out=90,in=180] (6.5,3.5) to [out=0,in=90] (10,0) node[below]{$\bar{6}$};
\draw[very thick] (4,0) node[below]{$5$} to [out=90,in=180] (8,4) to [out=0,in=90] (12,0) node[below]{$\bar{5}$};
\draw[very thick] (6,0) node[below]{$3$} to [out=90,in=180] (10,4) to [out=0,in=90] (14,0) node[below]{$\bar{3}$};
\draw[very thick] (8,0) node[below]{$4$} to [out=90,in=180] (10.5,2.5) to [out=0,in=90] (13,0) node[below]{$\bar{4}$};
\draw[very thick] (9,0) node[below]{$1$} to [out=90,in=180] (12.5,3.5) to [out=0,in=90] (16,0) node[below]{$\bar{1}$};
\draw[very thick] (11,0) node[below]{$2$} to [out=90,in=180] (13,2) to [out=0,in=90] (15,0) node[below]{$\bar{2}$};
\end{tikzpicture}
\implies
w(\mathcal{M}) = 7 \ 8 \ 6 \ 5 \ \bar{8} \ 3 \ \bar{7} \ 4 \ 1 \ \bar{6} \ 2 \ \bar{5} \ \bar{4} \ \bar{3} \ \bar{2} \ \bar{1} \, .
\end{equation}
We will refer to such words $w(\mathcal{M})$ as \emph{matching words} of length $2k$.
Clearly, there is a bijection between matchings and matching words, 
since given a matching word we can match the entries $i$ and $\bar{i}$.

The Sundaram--Stanley bijection associates
an \emph{oscillating tableau} $\UD(\mathcal{M})$ 
with $2k$ steps to each matching word $w(\mathcal{M})$ of length $2k$, 
and thus to each matching $\mathcal{M}$ of $2k$ points. 
Recall that a standard Young tableau (SYT) $Q$ with $k$ boxes can be thought of as a ``movie'' 
$(\lambda_t)_{1 \le t \le k}$ of partitions, where one box 
(the $t$-labeled box in $Q$) is added in each time interval. 
In contrast, an oscillating tableau $\UD$ with $2k$ steps is a sequence 
$(\UD_t)_{t=0}^{2k}$ of partitions wherein the Young diagram for
$\UD_{t+1}$ is obtained from that of $\UD_{t}$ by adding \textbf{or removing} one box. 
In the present paper, all oscillating tableaux will be assumed to satisfy 
$\UD_0 = \varnothing = \UD_{2k}$. 
While the movie $(\lambda_t)$ corresponding to an SYT can be recorded compactly by 
decorating the final time slice $\lambda = \lambda_k$ with numbers (producing $Q$), 
an oscillating tableaux can not be so compactly encoded (after all, the final partition $\UD_{2k}$ is empty). 
Thus, the notation for oscillating tableaux is more cumbersome 
(one must remember the whole movie rather than one time slice), 
even though the concept is not much more difficult.

The oscillating tableau $\UD(\mathcal{M})$ is defined by applying Schensted's bumping algorithm to 
the word $w(\mathcal{M})$ to add a box for each entry of the form $i$, and by simply removing 
the $i$-labeled box for each entry $\bar{i}$.
For example, for the word 
\[ 
w(\mathcal{M}) = 7 \ 8 \ 6 \ 5 \ \bar{8} \ 3 \ \bar{7} \ 4 \ 1 \ \bar{6} \ 2 \ \bar{5} \ \bar{4} \ \bar{3} \ \bar{2} \ \bar{1}
\]
from above, we have
\begin{equation}\label{eq:UDex}
\begin{aligned}
\varnothing \ , \ \
&
\begin{ytableau}
7 \cr
\end{ytableau} \ , \ \ 
\begin{ytableau}
7 & 8 \cr
\end{ytableau} \ , \ \
\begin{ytableau}
6 & 8 \cr
7 \cr
\end{ytableau} \ , \ \
\begin{ytableau}
5 & 8 \cr
6 \cr
7 \cr
\end{ytableau} \ , \ \
\begin{ytableau}
5 \cr
6 \cr
7 \cr
\end{ytableau} \ , \ \
\begin{ytableau}
3 \cr
5 \cr
6 \cr
7 \cr
\end{ytableau} \ , \ \
\begin{ytableau}
3 \cr
5 \cr
6 \cr
\end{ytableau} \ , \ \
\begin{ytableau}
3 & 4 \cr
5 \cr
6 \cr
\end{ytableau}
 \ , \ \ \\ \\
&
\begin{ytableau}
1 & 4 \cr
3 \cr
5 \cr
6 \cr
\end{ytableau} \ , \ \
\begin{ytableau}
1 & 4 \cr
3 \cr
5 \cr
\end{ytableau} \ , \ \
\begin{ytableau}
1 & 2 \cr
3 & 4 \cr
5 \cr
\end{ytableau} \ , \ \
\begin{ytableau}
1 & 2 \cr
3 & 4 \cr
\end{ytableau} \ , \ \
\begin{ytableau}
1 & 2 \cr
3 \cr
\end{ytableau} \ , \ \
\begin{ytableau}
1 & 2 \cr
\end{ytableau} \ , \ \
\begin{ytableau}
1 \cr
\end{ytableau} \ , \ \
\varnothing
\end{aligned}
\end{equation}
We emphasize that the numbers appearing in the Young diagrams in \eqref{eq:UDex} 
are shown here only to exhibit 
how to obtain $\UD(\mathcal{M})$ from $w(\mathcal{M})$,
and \emph{are not} part of the data of an oscillating tableau. The oscillating tableau simply consists of the 
sequence of Young diagrams. 
One can reconstruct the numbering\footnote{The first box removed is $k$, 
the next box removed is $k-1$, and so on. 
Reading the movie backwards, one can use the reverse bumping rule to determine 
how to update the tableau when a box is added.}, 
and hence the matching $\mathcal{M}$,
from the oscillating tableau $\UD(\mathcal{M})$, 
and we often keep the numbers for convenience.
More generally, the theorem of Sundaram--Stanley is as follows.

\begin{thm}\label{thm:SS}
The assignment $\mathcal{M} \mapsto \UD(\mathcal{M})$ determines a bijection 
between matchings of $2k$ points and oscillating tableaux with $2k$ steps.
\end{thm}

In \cite[Section 8]{SundaramThesis}, Sundaram establishes a version of this theorem. 
Specifically, she uses a variant of the assignment $\mathcal{M} \mapsto w(\mathcal{M})$ that 
instead labels the right endpoints in a matching diagram by the symbols $\bar{1},\ldots,\bar{k}$ 
from \emph{left-to-right}. The resulting bijection between such words and oscillating tableaux 
requires the use of jeu-de-taquin. We prefer the bijection defined above 
(which appears e.g.~in Stanley's work \cite[Section 9]{Stanley}) since it avoids the use of jeu-de-taquin
and directly generalizes the Robinson--Schensted correspondence.

Indeed, suppose we perform this operation on a matching word $\mathcal{M}$ coming from a permutation, 
like $2431 \bar{4} \bar{3} \bar{2} \bar{1}$. 
One obtains a special kind of oscillating tableau $\UD$, where all boxes are added before any are removed. 
If the permutation corresponds under Robinson--Schensted to $(P,Q,\lambda)$, 
then the boxes will be added according to the tableau $Q$
and then removed according to the tableau $P$ 
(as the tableau $P$ agrees with the numbers labeling $\UD_k$ in this case). 
That is, $Q$ corresponds to the movie $(\UD_t)_{t=0}^k$ 
and $P$ to the movie $(\UD_{2k-t})_{t=0}^k$. 
In this way, the oscillating tableau records the data of both $P$ and $Q$, 
hence the Sundaram--Stanley bijection is compatible with the Robinson--Schensted bijection.

\subsection{Decreasing sequences and patterns}\label{sec:intropattern}

When visualizing a permutation with its string diagram, 
a decreasing sequence of length $\ell$ is the same as a choice of $\ell$ strings for
which any pair of strands will cross each other. 
Said differently, given any subsequence of length $\ell$, 
one can view the $\ell$ corresponding strands as a string diagram for a permutation in $\SG_\ell$ 
by ignoring the other $k-\ell$ strands. 
The subsequence is then decreasing if and only if the permutation in $\SG_\ell$ is the longest element, 
also known as the \emph{half twist}.

Similar ideas apply to matchings of $2k$ points. 
Let us say that an \emph{$\ell$-pattern} (or a \emph{pattern of length $\ell$}) in a matching $\mathcal{M}$ 
is a choice of $\ell$ strings for which any pair of strands will cross one another. 
This corresponds to the appearance of a half twist inside the matching diagram of the matching. 
From the perspective of matching words, an $\ell$-pattern corresponds to a decreasing subsequence 
of un-barred symbols that concludes before any of their barred versions appears. 
We will refer to such a subsequence as a \emph{coexistent decreasing subsequence}.

For example, consider the matching word
\[ 
w(\mathcal{M}) = \mathbf{7} \ 8 \ \mathbf{6} \ \mathbf{5} \ \bar{8} \ \mathbf{3} \ 
	{\color{red} \bar{7}} \ 4 \ 1 \ \bar{6} \ 2 \ \bar{5} \ \bar{4} \ \bar{3} \ \bar{2} \ \bar{1}
\]
from \eqref{eq:matchingex}.
The bold subsequence $7653$ is a coexistent decreasing subsequence, 
hence corresponds to a $4$-pattern in $\mathcal{M}$. 
On the other hand, the decreasing subsequence $76531$ is not coexistent,
since $1$ does not appear until after $\bar{7}$. 
In the matching diagram \eqref{eq:matchingex}, one can see that the $7$-labeled string 
and the $1$-labeled string do not cross. 
In fact, this matching does not contain a $5$-pattern.

In this paper we prove the following, which is a generalization of Theorem \ref{thm:longest}.

\begin{thm}\label{thm:longestC}
Let $\mathcal{M}$ be a matching of $2k$ points. 
The length of the longest pattern in $\mathcal{M}$ is equal to the maximum number of rows 
appearing in any partition in $\UD(\mathcal{M})$. 
\end{thm}

This result, which in \cite{Kup,RubWes} is attributed to Sundaram, 
is stated in \cite[Theorem 16]{Stanley} without proof. 
A proof can be deduced from the more-general results on set partitions and ``vacillating tableau''
in \cite{CDDSY}; see Section 5 therein, and also interesting
work of Krattenthaler \cite{Krattenthaler} that re\"{e}stablishes the results in \cite{CDDSY} 
using Fomin's growth diagrams \cite{Fomin}.

Our proof of Theorem \ref{thm:longestC} in \S \ref{sec:updown} is based on 
a simple analogue of Viennot's geometric construction for oscillating tableaux.
By mimicking Viennot's proof of Theorem \ref{thm:longest} with one small modification, 
we prove that a matching has an $\ell$-pattern if and only if some partition in its oscillating tableau
has at least $\ell$ rows. We also give an easy algorithm to find this $\ell$-pattern.

\begin{rem}
As pointed out by the referee, 
an extension of Viennot's construction to skew oscillating tableaux is given in 
earlier work of Chauve--Dulucq \cite{ChauveDulucq}, 
but without the connection to decreasing subsequences.
In particular, they do not discuss Theorem \ref{thm:longestC}.
In the non-skew case we consider, their Viennot diagrams differ from ours
both in their construction and (slightly) in their final form;
however, they contain essentially the same information.
Our construction of the Viennot diagram encodes a timeline 
of the bumping algorithm for $w(\mathcal{M})$,
which the reader may find somewhat more straightforward.
	\end{rem}

\begin{rem}[Relation to $\spn$ representation theory]\label{rem:spn}
We conclude the introduction by discussing the connection to representation theory 
that motivated our interest in the Sundaram--Stanley bijection.

Consider the vector representation $V = \C^{2n}$ of $\spn$. 
Brauer--Schur--Weyl duality gives a surjective $\C$-algebra homomorphism
\begin{equation}\label{eq:BSW}
\BSW \colon 
\Brauer_k \to \End_{\spn}\big( V^{\otimes k} \big) \, .
\end{equation}
The Brauer algebra $\Brauer_k$ (and its quantum analogue the BMW algebra \cite{Murakami,BW})
is typically described diagrammatically using tangle diagrams, 
and it has a basis in bijection with matchings of $2k$ points. 
The kernel of the homomorphism $\BSW$ is less easy to describe than in type $A$;
see \cite[Section 2.3]{BERT} for some discussion and history. 
Nonetheless, as proved by Sundaram in \cite[Theorem 6.15]{SundaramThesis}, 
the dimension of the endomorphism ring on the right-hand side of \eqref{eq:BSW} 
(equivalently, the dimension of the space of $\spn$-invariant vectors in $V^{\otimes 2k}$)
precisely matches the number of oscillating tableaux such that all partitions appearing have $n$ or fewer rows.
By Theorem \ref{thm:longestC}, this agrees with the number of matchings of $2k$ points that avoid $(n+1)$-patterns.

In \cite{BERT}, we present (the image of) $\Brauer_k$ acting on $V^{\otimes k}$ (and its quantum analogue)
using matchings; see the category $\Web^\times(\spn)$ from Section 5.4 therein.
In \cite[Theorem 5.35]{BERT}, we show that $\Web^\times(\spn)$ has a spanning set 
consisting of matchings that avoid $(n+1)$-patterns. 
Using Theorem \ref{thm:longestC} and Sundaram's result, 
we are able to prove that the relevant $\Hom$-space in 
$\Web^\times(\spn)$ is isomorphic to $\End_{\spn}\big( V^{\otimes k} \big)$.
This was the reason for our interest in Theorem \ref{thm:longestC}.
\end{rem}

\begin{ack}
E.B. and B.E. were supported on this project by NSF CAREER grant DMS-1553032. 
B.E. was also supported by RTG grant DMS-2039316,
and by a stay at the Institute of Advanced Study, supported via NSF grant DMS-1926686. 
During revision, B.E. was supported by DMS-2201387.
D.E.V.R. and L.T. were partially supported by Simons Collaboration Grant 523992: 
\emph{Research on knot invariants, representation theory, and categorification}
and D.E.V.R. was also partially supported by NSF CAREER grant DMS-2144463.
The authors would like to thank Christian Krattenthaler for bringing our attention to \cite{Krattenthaler},
which provides an alternative route to Theorem \ref{thm:longestC} via Fomin's growth diagrams.
We also thank an anonymous referee for pointers to the literature and for their help in streamlining the presentation, 
especially of the material in \S \ref{sec:Viennot}.
\end{ack}

%
\section{Viennot's geometric construction and decreasing subsequences}\label{sec:Viennot}
%

In \cite{Viennot}, Viennot gives a graphical interpretation of the Robinson--Schensted correspondence for $\SG_k$
in terms of certain colored paths in the first quadrant
$\{(t,b) \mid t,b \geq 0\} \subset \R^2$.
We now recall the construction as presented in \cite[Section 3.6]{Sagan},
working with the running example of the permutation $w=2 \ 4 \ 3 \ 1$ in $\SG_4$.
The procedure is iterative: at each step, it starts with a collection of points 
in the first quadrant and outputs a \emph{shadow diagram}: 
a collection of non-intersecting lattice paths, called \emph{shadow lines},
that weakly decrease in the $b$-direction and weakly increase in the $t$-direction. 
The $i$th shadow diagram then determine the input points for the $(i+1)$st step.

To start, one plots the points $(j, w(j))$ for $1 \leq j \leq k$ and considers the 
region\footnote{Technically, Viennot considers unions of products of rays 
 $[j,\infty) \times [w(j),\infty)$, but the portion of this region lying outside the square
 $[0,k+1] \times [0,k+1]$ has no bearing on the construction.}
\[S_1=\bigcup_{j=1}^k [j,k+1] \times [w(j),k+1] \, . \]
The boundary of this region determines the first shadow line associated with this $1$st step, 
e.g.:
\[
\begin{tikzpicture}[anchorbase, scale=.5,smallnodes]
\foreach \x in {1,...,4}
{\draw (\x,0) node[below]{$\x$} to (\x,5);}
\foreach \y in {1,...,4}
{\draw (0,\y) node[left]{$\y$} to (5,\y);}
\draw[thick,->] (0,0) to (5,0) node[right]{$t$};
\draw[thick,->] (0,0) to (0,5) node[above]{$b$};
\fill[gray,opacity=.25] (1,5) to (1,2) to (5,2) to (5,5);
\fill[gray,opacity=.25] (2,5) to (2,4) to (5,4) to (5,5);
\fill[gray,opacity=.25] (3,5) to (3,3) to (5,3) to (5,5);
\fill[gray,opacity=.25] (4,5) to (4,1) to (5,1) to (5,5);
\draw[red, very thick] (1,5) to (1,2) to (4,2) to (4,1) to (5,1);
\node at (1,2){$\bullet$};
\node at (2,4){$\bullet$};
\node at (3,3){$\bullet$};
\node at (4,1){$\bullet$};
\end{tikzpicture}
\]
Next, one considers the subset $S_2 \subset S_1$ 
given as the union of the subsets $[j,k+1] \times [w(j),k+1] \subset S_1$ 
that do not intersect the first shadow line. 
The boundary of this region is the second shadow line associated 
with the $1$st step, e.g.
\begin{equation}\label{eq:V1}
\begin{tikzpicture}[anchorbase, scale=.5,smallnodes]
\foreach \x in {1,...,4}
{\draw (\x,0) node[below]{$\x$} to (\x,5);}
\foreach \y in {1,...,4}
{\draw (0,\y) node[left]{$\y$} to (5,\y);}
\draw[thick,->] (0,0) to (5,0) node[right]{$t$};
\draw[thick,->] (0,0) to (0,5) node[above]{$b$};
\fill[gray,opacity=.25] (2,5) to (2,4) to (5,4) to (5,5);
\fill[gray,opacity=.25] (3,5) to (3,3) to (5,3) to (5,5);
\draw[red, very thick] (1,5) to (1,2) to (4,2) to (4,1) to (5,1);
\draw[red, very thick] (2,5) to (2,4) to (3,4) to (3,3) to (5,3);
\node at (1,2){$\bullet$};
\node at (2,4){$\bullet$};
\node at (3,3){$\bullet$};
\node at (4,1){$\bullet$};
\end{tikzpicture}
\end{equation}
One then repeats, considering the subset $S_3 \subset S_2$ consisting of the subsets $[j,k+1] \times [w(j),k+1] \subset S_1$ 
that do not intersect the second shadow line, and use this to draw the third line. 
One repeats in this manner until $S_\ell = \varnothing$, at which point the $1$st step is complete and we have constructed 
the $1$st shadow diagram. 
(In our running example, this is simply the union of the two {\color{red}\textbf{red}} paths in \eqref{eq:V1}.)

All subsequent steps are identical to the $1$st, 
but with the starting collection of points
determined by the shadow diagram from the previous step. 
Specifically, observe that every corner in a shadow line takes the form
\[
\oc \colon
\begin{tikzpicture}[anchorbase, scale=.5,smallnodes]
\foreach \x in {1,...,4}
{\draw[gray] (\x,1) to (\x,4);}
\foreach \y in {1,...,4}
{\draw[gray] (1,\y) to (4,\y);}
\draw[very thick] (3,2) to (2,2) to (2,3);
\end{tikzpicture}
\quad \text{or} \quad
\ic\colon
\begin{tikzpicture}[anchorbase, scale=.5,smallnodes]
\foreach \x in {1,...,4}
{\draw[gray] (\x,1) to (\x,4);}
\foreach \y in {1,...,4}
{\draw[gray] (1,\y) to (4,\y);}
\draw[very thick] (3,2) to (3,3) to (2,3);
\end{tikzpicture} \, .
\]
We call these \emph{outer corners} and \emph{inner corners}, respectively. 
(In \cite{Sagan}, inner corners are called \emph{northeast corners}, 
while outer corners are not given a name.)
For $i \geq 1$, the starting collection of points for the $(i+1)$st step are the inner corners of the 
shadow diagram from the $i$th step.
For example, the $1$st, $2$nd and $3$rd (last) steps in our running example 
produce the following shadow diagrams
(with inner corners marked by $\circ$):
\[
\begin{tikzpicture}[anchorbase, scale=.5,smallnodes]
\foreach \x in {1,...,4}
{\draw (\x,0) node[below]{$\x$} to (\x,5);}
\foreach \y in {1,...,4}
{\draw (0,\y) node[left]{$\y$} to (5,\y);}
\draw[thick,->] (0,0) to (5,0) node[right]{$t$};
\draw[thick,->] (0,0) to (0,5) node[above]{$b$};
\draw[red, very thick] (1,5) to (1,2) to (4,2) to (4,1) to (5,1);
\draw[red, very thick] (2,5) to (2,4) to (3,4) to (3,3) to (5,3);
\node at (1,2){$\bullet$};
\node at (2,4){$\bullet$};
\node at (3,3){$\bullet$};
\node at (4,1){$\bullet$};
\node at (3,4){$\circ$};
\node at (4,2){$\circ$};
\end{tikzpicture}
\qquad , \qquad
\begin{tikzpicture}[anchorbase, scale=.5,smallnodes]
\foreach \x in {1,...,4}
{\draw (\x,0) node[below]{$\x$} to (\x,5);}
\foreach \y in {1,...,4}
{\draw (0,\y) node[left]{$\y$} to (5,\y);}
\draw[thick,->] (0,0) to (5,0) node[right]{$t$};
\draw[thick,->] (0,0) to (0,5) node[above]{$b$};
\draw[orange, very thick] (3,5) to (3,4) to (4,4) to (4,2) to (5,2);
\node at (3,4){$\bullet$};
\node at (4,2){$\bullet$};
\node at (4,4){$\circ$};
\end{tikzpicture}
\qquad , \qquad
\begin{tikzpicture}[anchorbase, scale=.5,smallnodes]
\foreach \x in {1,...,4}
{\draw (\x,0) node[below]{$\x$} to (\x,5);}
\foreach \y in {1,...,4}
{\draw (0,\y) node[left]{$\y$} to (5,\y);}
\draw[thick,->] (0,0) to (5,0) node[right]{$t$};
\draw[thick,->] (0,0) to (0,5) node[above]{$b$};
\draw[green, very thick] (4,5) to (4,4) to (5,4);
\node at (4,4){$\bullet$};
\end{tikzpicture} \, .
\]
As in our running example, we will choose a totally ordered set of colors 
$c_1 < c_2 < \cdots$ and color the $i$th shadow diagram with the color $c_i$.

\begin{defn}[Viennot's geometric construction]\label{def:Viennot1}
Let $w \in \SG_k$ be a permutation written in one-line notation as $w = a_1 \cdots a_k$, 
and let $\CS_k=\{c_1 < \cdots < c_k\}$ be a totally ordered set of $k$ colors. 
The \emph{Viennot diagram} of $w$ is the union of its shadow diagrams, 
with the $i$th diagram colored $c_i$.
	\end{defn}

\begin{example}\label{ex:SaganV}
For our running example permutation $w=2 \ 4 \ 3 \ 1$ in $\SG_4$ and ordering of colors
{\color{red}\textbf{red}} $<$ {\color{myorange}\textbf{orange}} $<$ {\color{green} \textbf{green}} 
($<$ {\color{blue}\textbf{blue}}), the Viennot diagram is:
\[
\begin{tikzpicture}[anchorbase, scale=.5,smallnodes]
\foreach \x in {1,...,4}
{\draw (\x,0) node[below]{$\x$} to (\x,5);}
\foreach \y in {1,...,4}
{\draw (0,\y) node[left]{$\y$} to (5,\y);}
\draw[thick,->] (0,0) to (5,0) node[right]{$t$};
\draw[thick,->] (0,0) to (0,5) node[above]{$b$};
\draw[very thick, red] (1,2) to (1,5);
\draw[very thick, red] (1,2) to (4,2); \draw[very thick, myorange] (4,2) to (5,2); 
\node at (1,2){$\bullet$};
\draw[very thick, red] (2,4) to (2,5);
\draw[very thick, red] (2,4) to (3,4); \draw[very thick, myorange] (3,4) to (4,4); 
	\draw[very thick, green] (4,4) to (5,4);  
\node at (2,4){$\bullet$};
\draw[very thick, red] (3,3) to (3,4); \draw[very thick, myorange] (3,4) to (3,5); 
\draw[very thick, red] (3,3) to (5,3);
\node at (3,3){$\bullet$};
\draw[very thick, red] (4,1) to (4,2); \draw[very thick, myorange] (4,2) to (4,4); 
	\draw[very thick, green] (4,4) to (4,5);  
\draw[very thick, red] (4,1) to (5,1);
\node at (4,1){$\bullet$};
\end{tikzpicture}
\]
Observe that all $k$ colors need not appear in a Viennot diagram.
\end{example}

\begin{example}\label{ex:big}
The following is the Viennot diagram for the permutation
\[
w=2 \ 9 \ 1 \ 15 \ 4 \ 7 \ 13 \ 18 \ 11 \ 19 \ 5 \ 14 \ 3 \ 10 \ 6 \ 17 \ 8 \ 16 \ 12
\]
in $\SG_{19}$. 
Our colors are 
{\color{red}\textbf{red}} $<$ {\color{myorange}\textbf{orange}} $<$ {\color{green} \textbf{green}} 
$<$ {\color{blue}\textbf{blue}} $<$ 
{\textbf{black}}  ($< \cdots$).
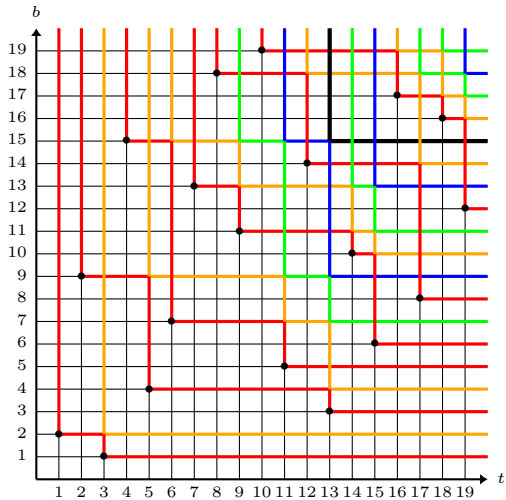
\[
\begin{tikzpicture}[anchorbase, scale=.3,tinynodes]
\foreach \x in {1,...,19}
{\draw (\x,0) node[below]{$\x$} to (\x,20);}
\foreach \y in {1,...,19}
{\draw (0,\y) node[left]{$\y$} to (20,\y);}
\draw[thick,->] (0,0) to (20,0) node[right]{$t$};
\draw[thick,->] (0,0) to (0,20) node[above]{$b$};
\draw[very thick, red] (1,2) to (1,20);
\draw[very thick, red] (1,2) to (3,2); \draw[very thick, myorange] (3,2) to (20,2);
\node at (1,2){$\bullet$};
\draw[very thick, red] (2,9) to (2,20);
\draw[very thick, red] (2,9) to (5,9); \draw[very thick, myorange] (5,9) to (11,9);
	\draw[very thick, green] (11,9) to (13,9); \draw[very thick, blue] (13,9) to (20,9);
\node at (2,9){$\bullet$};
\draw[very thick, red] (3,1) to (3,2); \draw[very thick, myorange] (3,2) to (3,20);
\draw[very thick, red] (3,1) to (20,1);
\node at (3,1){$\bullet$};
\draw[very thick, red] (4,15) to (4,20);
\draw[very thick, red] (4,15) to (6,15); \draw[very thick, myorange] (6,15) to (9,15);
	\draw[very thick, green] (9,15) to (11,15); \draw[very thick, blue] (11,15) to (13,15);
		\draw[ultra thick] (13,15) to (20,15);
\node at (4,15){$\bullet$};
\draw[very thick, red] (5,4) to (5,9); \draw[very thick, myorange] (5,9) to (5,20);
\draw[very thick, red] (5,4) to (13,4); \draw[very thick, myorange] (13,4) to (20,4);
\node at (5,4){$\bullet$};
\draw[very thick, red] (6,7) to (6,15); \draw[very thick, myorange] (6,15) to (6,20);
\draw[very thick, red] (6,7) to (11,7); \draw[very thick, myorange] (11,7) to (13,7);
	\draw[very thick, green] (13,7) to (20,7);
\node at (6,7){$\bullet$};
\draw[very thick, red] (7,13) to (7,20);
\draw[very thick, red] (7,13) to (9,13); \draw[very thick, myorange] (9,13) to (14,13);
	\draw[very thick, green] (14,13) to (15,13); \draw[very thick, blue] (15,13) to (20,13);
\node at (7,13){$\bullet$};
\draw[very thick, red] (8,18) to (8,20);
\draw[very thick, red] (8,18) to (12,18); \draw[very thick, myorange] (12,18) to (17,18);
	\draw[very thick, green] (17,18) to (19,18); \draw[very thick, blue] (19,18) to (20,18);
\node at (8,18){$\bullet$};
\draw[very thick, red] (9,11) to (9,13); \draw[very thick, myorange] (9,13) to (9,15);
	\draw[very thick, green] (9,15) to (9,20);
\draw[very thick, red] (9,11) to (14,11); \draw[very thick, myorange] (14,11) to (15,11);
	\draw[very thick, green] (15,11) to (20,11);
\node at (9,11){$\bullet$};
\draw[very thick, red] (10,19) to (10,20);
\draw[very thick, red] (10,19) to (16,19); \draw[very thick, myorange] (16,19) to (18,19);
	\draw[very thick, green] (18,19) to (20,19);
\node at (10,19){$\bullet$};
\draw[very thick, red] (11,5) to (11,7); \draw[very thick, myorange] (11,7) to (11,9);
	\draw[very thick, green] (11,9) to (11,15); \draw[very thick, blue] (11,15) to (11,20);
\draw[very thick, red] (11,5) to (20,5);
\node at (11,5){$\bullet$};
\draw[very thick, red] (12,14) to (12,18); \draw[very thick, myorange] (12,18) to (12,20);
\draw[very thick, red] (12,14) to (17,14); \draw[very thick, myorange] (17,14) to (20,14);
\node at (12,14){$\bullet$};
\draw[very thick, red] (13,3) to (13,4); \draw[very thick, myorange] (13,4) to (13,7);
	\draw[very thick, green] (13,7) to (13,9); \draw[very thick, blue] (13,9) to (13,15);
		\draw[ultra thick] (13,15) to (13,20);
\draw[very thick, red] (13,3) to (20,3);
\node at (13,3){$\bullet$};
\draw[very thick, red] (14,10) to (14,11); \draw[very thick, myorange] (14,11) to (14,13);
	\draw[very thick, green] (14,13) to (14,20);
\draw[very thick, red] (14,10) to (15,10); \draw[very thick, myorange] (15,10) to (20,10);
\node at (14,10){$\bullet$};
\draw[very thick, red] (15,6) to (15,10); \draw[very thick, myorange] (15,10) to (15,11);
	\draw[very thick, green] (15,11) to (15,13); \draw[very thick, blue] (15,13) to (15,20);
\draw[very thick, red] (15,6) to (20,6);
\node at (15,6){$\bullet$};
\draw[very thick, red] (16,17) to (16,19); \draw[very thick, myorange] (16,19) to (16,20);
\draw[very thick, red] (16,17) to (18,17); \draw[very thick, myorange] (18,17) to (19,17);
	\draw[very thick, green] (19,17) to (20,17);
\node at (16,17){$\bullet$};
\draw[very thick, red] (17,8) to (17,14); \draw[very thick, myorange] (17,14) to (17,18);
	\draw[very thick, green] (17,18) to (17,20);
\draw[very thick, red] (17,8) to (20,8);
\node at (17,8){$\bullet$};
\draw[very thick, red] (18,16) to (18,17); \draw[very thick, myorange] (18,17) to (18,19);
	\draw[very thick, green] (18,19) to (18,20);
\draw[very thick, red] (18,16) to (19,16); \draw[very thick, myorange] (19,16) to (20,16);
\node at (18,16){$\bullet$};
\draw[very thick, red] (19,12) to (19,16); \draw[very thick, myorange] (19,16) to (19,17);
	\draw[very thick, green] (19,17) to (19,18); \draw[very thick, blue] (19,18) to (19,20);
\draw[very thick, red] (19,12) to (20,12);
\node at (19,12){$\bullet$};
\end{tikzpicture}
\]
\end{example}

\smallskip

The connection to the Robinson--Schensted correspondence is the following result of Viennot.

\begin{thm}[{\cite{Viennot}}]\label{thm:V}
Let $w \in \SG_k$ be a permutation and consider the Viennot diagram assigned to $w$. 
Then, $\RS(w) = (P,Q,\lambda)$, where $P$ is determined by the colors appearing 
on the right of the Viennot diagram, and $Q$ is determined by the colors appearing on top. 
More precisely, the $j^{th}$ row of $P$ is filled with the $b$-coordinates of the $c_j$-colored 
shadow lines meeting the line $t=k+1$. 
Similarly, the $j^{th}$ row of $Q$ is filled with the $t$-coordinates of the $c_j$-colored 
shadow lines meeting the line $b = k+1$. \qed
\end{thm}

\begin{example} For the permutation in Example \ref{ex:big}, the first row of $P$ has entries $(1,3,5,6,8,12)$, 
coming from the right ends of the {\color{red}\textbf{red}} shadow lines. 
The second row of $P$ has entries $(2,4,10,14,16)$, 
coming from the {\color{myorange}\textbf{orange}} shadow lines. 
Since $5$ colors appear in the diagram, $P$ has $5$ nonempty rows. 
Meanwhile, the first row of $Q$ has entries $(1,2,4,7,8,10)$, 
coming from the tops of the {\color{red}\textbf{red}} shadow lines. 
\end{example}

We now recall the connection between Viennot's construction and decreasing subsequences.
We call a sequence $(t_1,b_1) , \ldots, (t_r,b_r)$ of points in a Viennot diagram a
\emph{length $r$ decreasing sequence} if $t_1 < \cdots < t_r$ and $b_1 > \cdots >b_r$.
The following is essentially a retooling
of Viennot's argument for the decreasing subsequence case of Theorem \ref{thm:longest}.

\begin{lem} \label{lem:algorithm}
A Viennot diagram contains a length $r$ decreasing sequence of 
$c$-colored inner corners if and only if it contains a length $r+1$ decreasing sequence of 
$c$-colored outer corners. \end{lem}

\begin{proof}
Note that each $c$-colored inner corner 
$\ic = (t,b)$ 
determines two associated $c$-colored outer corners 
$l(\ic) = (l(t),b)$ and $d(\ic) = (t,d(b))$
with 
$l(t) < t$ and $d(b) < b$,
which are obtained by following the shadow line meeting $\ic$ $l$eftward and $d$ownward 
until they arrive at outer corners. 
These outer corners necessarily exist, 
since the $c$-colored shadow line never leaves the first quadrant, 
thus cannot travel leftward or downward from an inner corner indefinitely.

Now, suppose that a Viennot diagram contains a length $r$ decreasing sequence 
$\ic_1 , \ldots , \ic_r$ of $c$-colored inner corners, 
and write $\ic_j = (t_j,b_j)$.
For $1 \leq j \leq r+1$, consider the following subsets:
\[
B_j := 
\begin{cases}
(-\infty,t_1) \times [b_1,\infty) & \text{if } j=1 \\
[t_{j-1},t_j) \times [b_j,b_{j-1}) & 2 \leq j \leq r \\
[t_r,\infty) \times (-\infty,b_r) & j=r+1 \, .
\end{cases} 
\]
Aside from the edge cases $j = 1$ and $j = r+1$, 
$B_j$ is the rectangle between two consecutive corners in the decreasing sequence. 
In Example \ref{ex:grayboxes} below, one can see the rectangles $B_j$ shaded in gray. 
By construction, the region $B_{j+1}$ lies completely below and to the right of the region $B_j$. 
Thus any collection of $r+1$ points with each point lying 
in a distinct region $B_j$ will yield a decreasing sequence.

It suffices to show that each $B_j$ contains at least one $c$-colored outer corner. 
For the edge cases, we must have that $l(\ic_1) \in B_0$ and $d(\ic_r) \in B_{r+1}$. 
For $2 \le j \le r$, since distinct shadow lines in a shadow diagram never intersect, 
at least one of $d(\ic_{j-1})$ or $l(\ic_j)$ lies in $B_j$, because the
situation must fit one of the following schematics\footnote{It is a game of chicken. 
A vehicle leaves $\ic_j$ traveling left, and another leaves $\ic_{j-1}$ traveling down, 
on a deadly collision course! Either they collide (please fetch the outer cor(o)ner), 
or (at least) one of the vehicles is a chicken, veering off before the potential collision.}: 
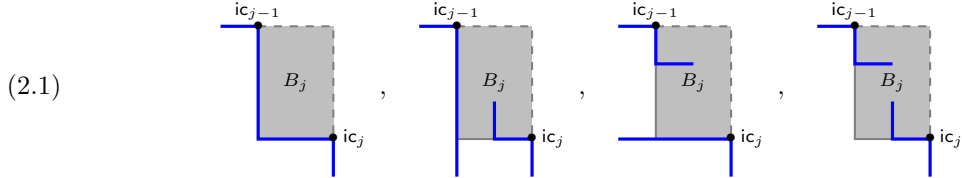
\begin{equation}\label{eq:chicken} 
\begin{tikzpicture}[anchorbase, scale=.5,smallnodes] \path[fill=gray,opacity=.5] (0,3)
rectangle (2,0); \node at (1,1.5){$B_j$}; \draw[thick,gray] (0,3) node[above,black]{$\ic_{j-1}$} to (0,0) to (2,0) node[right,black]{$\ic_{j}$}; \draw[thick,gray,dashed] (0,3) to
(2,3) to (2,0);
\draw[very thick,blue] (-1,3) to (0,3) to (0,0) to (2,0) to (2,-1);
\node at (2,0){$\bullet$};
\node at (0,3){$\bullet$};  
\end{tikzpicture}
\, , \quad
\begin{tikzpicture}[anchorbase, scale=.5,smallnodes]
\path[fill=gray,opacity=.5] (0,3) rectangle (2,0);
\node at (1,1.5){$B_j$};
\draw[thick,gray] (0,3) node[above,black]{$\ic_{j-1}$} to (0,0) 
	to (2,0) node[right,black]{$\ic_{j}$};
\draw[thick,gray,dashed] (0,3) to (2,3) to (2,0);
\draw[very thick,blue] (-1,3) to (0,3) to (0,-1);
\draw[very thick,blue] (2,-1) to (2,0) to (1,0) to (1,1);
\node at (2,0){$\bullet$};
\node at (0,3){$\bullet$};  
\end{tikzpicture}
\, , \quad
\begin{tikzpicture}[anchorbase, scale=.5,smallnodes]
\path[fill=gray,opacity=.5] (0,3) rectangle (2,0);
\node at (1,1.5){$B_j$};
\draw[thick,gray] (0,3) node[above,black]{$\ic_{j-1}$} to (0,0) 
	to (2,0) node[right,black]{$\ic_{j}$};
\draw[thick,gray,dashed] (0,3) to (2,3) to (2,0);
\draw[very thick,blue] (-1,3) to (0,3) to (0,2) to (1,2);
\draw[very thick,blue] (2,-1) to (2,0) to (-1,0);
\node at (2,0){$\bullet$};
\node at (0,3){$\bullet$};  
\end{tikzpicture}
\, , \quad
\begin{tikzpicture}[anchorbase, scale=.5,smallnodes]
\path[fill=gray,opacity=.5] (0,3) rectangle (2,0);
\node at (1,1.5){$B_j$};
\draw[thick,gray] (0,3) node[above,black]{$\ic_{j-1}$} to (0,0) 
	to (2,0) node[right,black]{$\ic_{j}$};
\draw[thick,gray,dashed] (0,3) to (2,3) to (2,0);
\draw[very thick,blue] (-1,3) to (0,3) to (0,2) to (1,2);
\draw[very thick,blue] (2,-1) to (2,0) to (1,0) to (1,1);
\node at (2,0){$\bullet$};
\node at (0,3){$\bullet$};  
\end{tikzpicture}
\end{equation}
This establishes one direction of Lemma \ref{lem:algorithm}.

The other direction is entirely analogous. 
Let $\oc_1 , \ldots , \oc_{r+1}$ be a decreasing sequence of 
$c$-colored outer corners, and write $\oc_j = (s_j,u_j)$. 
If we consider the subsets
\[
B^j := (s_j,s_{j+1}] \times (u_{j+1},u_j]
\]
for $1 \leq j \leq r$, then each such box must contain a $c$-colored inner corner 
(the schematic is obtained from the one in \eqref{eq:chicken} by rotating each picture $180$ degrees).
Choosing one such inner corner in each $B^j$ gives the desired decreasing sequence.
\end{proof}

\begin{prop}\label{prop:corners}
Let $w = a_1 \cdots a_k$ be a permutation in $\SG_k$ written in one-line notation, 
and let $c_1 < \cdots < c_k$ 
be the totally ordered set of colors in Viennot's geometric construction. 
The color $c_\ell$ appears in the Viennot diagram for $w$ if and only if
the list $a_1, \ldots, a_k$ contains a decreasing subsequence of length $\ell$.
\end{prop}

\begin{proof}
To begin, we record the following observations concerning a Viennot diagram:
\begin{enumerate}
\item For $i\ge 1$, each $c_{i+1}$-colored outer corner meets exactly one $c_{i}$-colored inner corner, 
and vice versa:
\[
\begin{tikzpicture}[anchorbase, scale=.5,smallnodes]
\foreach \x in {2,...,4}
{\draw[gray] (\x,2) to (\x,4);}
\foreach \y in {2,...,4}
{\draw[gray] (2,\y) to (4,\y);}
\draw[very thick, green] (3,2) node[below]{$c_i$} to (3,3) to (2,3);
\draw[very thick, blue] (4,3) node[right]{$c_{i+1}$} to (3,3) to (3,4);
\end{tikzpicture}
\]
\item Every shadow line contains at least one outer corner.
\item The union of the $c_1$-colored outer corners is equal to the set $\{(t,a_t) \mid 1 \leq t \leq k\}$.
\end{enumerate}

By (1), a length $r$ decreasing sequence of $c_i$-colored outer corners is 
also a length $r$ decreasing sequence of $c_{i-1}$-colored inner corners. 
Thus, Lemma \ref{lem:algorithm} allows us to construct a length $r+1$ decreasing sequence of $c_{i-1}$-colored outer 
corners from a length $r$ decreasing sequence of $c_i$-colored outer corners, 
\emph{and} vice-versa.

If the color $c_\ell$ appears, by (2) we can choose one $c_\ell$-colored outer corner 
and thus iteratively construct a length $\ell$ decreasing sequence of $c_1$-colored outer corners.
By (3), this sequence of outer corners yields the desired decreasing subsequence.
Conversely, given a length $\ell$ decreasing sequence of $c_1$-colored outer corners, 
we can run the procedure in reverse to obtain one $c_\ell$-colored outer corner, 
implying that the color $c_{\ell}$ appears in the diagram.
\end{proof}

\begin{example} \label{ex:grayboxes}
We illustrate the passage from inner to outer corners in Lemma \ref{lem:algorithm} using the 
Viennot diagram from Example \ref{prop:corners}. 
The decreasing sequence of {\color{myorange}\textbf{orange}} inner corners 
determines the indicated decreasing sequence of outer corners.
(Here, we've removed all shadow lines except those that are {\color{myorange}\textbf{orange}}.)
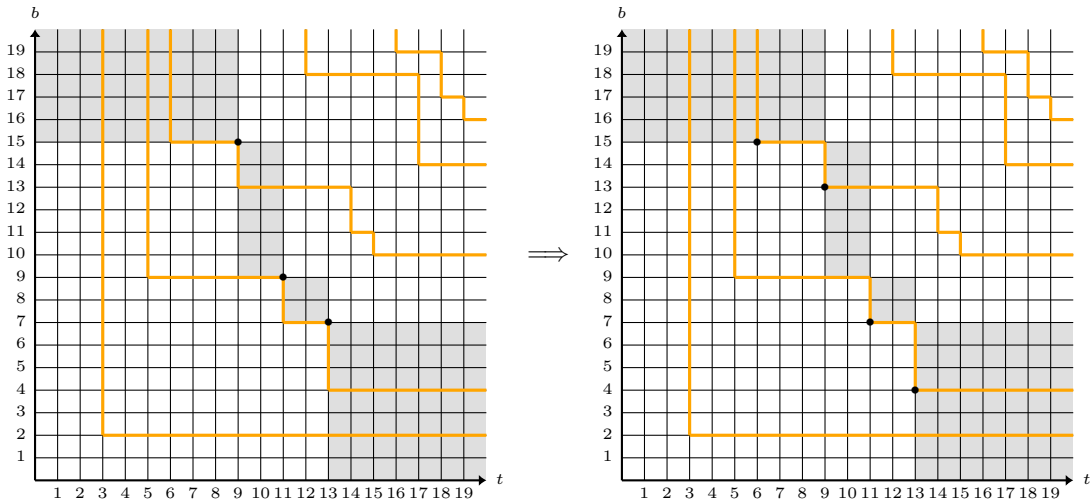
\[
\begin{tikzpicture}[anchorbase, scale=.3,tinynodes]
\path[fill=gray,opacity=.25] (0,20) rectangle (9,15);
\path[fill=gray,opacity=.25] (9,15) rectangle (11,9);
\path[fill=gray,opacity=.25] (11,9) rectangle (13,7);
\path[fill=gray,opacity=.25] (13,7) rectangle (20,0);
\foreach \x in {1,...,19}
{\draw (\x,0) node[below]{$\x$} to (\x,20);}
\foreach \y in {1,...,19}
{\draw (0,\y) node[left]{$\y$} to (20,\y);}
\draw[thick,->] (0,0) to (20,0) node[right]{$t$};
\draw[thick,->] (0,0) to (0,20) node[above]{$b$};
\draw[very thick, myorange] (3,2) to (20,2);
\draw[very thick, myorange] (5,9) to (11,9);
\draw[very thick, myorange] (3,2) to (3,20);
\draw[very thick, myorange] (6,15) to (9,15);
\draw[very thick, myorange] (5,9) to (5,20);
\draw[very thick, myorange] (13,4) to (20,4);
\draw[very thick, myorange] (6,15) to (6,20);
\draw[very thick, myorange] (11,7) to (13,7);
\draw[very thick, myorange] (9,13) to (14,13);
\draw[very thick, myorange] (12,18) to (17,18);
\draw[very thick, myorange] (9,13) to (9,15);
\draw[very thick, myorange] (14,11) to (15,11);
\draw[very thick, myorange] (16,19) to (18,19);
\draw[very thick, myorange] (11,7) to (11,9);
\draw[very thick, myorange] (12,18) to (12,20);
\draw[very thick, myorange] (17,14) to (20,14);
\draw[very thick, myorange] (13,4) to (13,7);
\draw[very thick, myorange] (14,11) to (14,13);
\draw[very thick, myorange] (15,10) to (20,10);
\draw[very thick, myorange] (15,10) to (15,11);
\draw[very thick, myorange] (16,19) to (16,20);
\draw[very thick, myorange] (18,17) to (19,17);
\draw[very thick, myorange] (17,14) to (17,18);
\draw[very thick, myorange] (18,17) to (18,19);
\draw[very thick, myorange] (19,16) to (20,16);
\draw[very thick, myorange] (19,16) to (19,17);
\node at (9,15) {$\bullet$}; \node at (11,9) {$\bullet$}; \node at (13,7) {$\bullet$};
\end{tikzpicture}
\implies
\begin{tikzpicture}[anchorbase, scale=.3,tinynodes]
\path[fill=gray,opacity=.25] (0,20) rectangle (9,15);
\path[fill=gray,opacity=.25] (9,15) rectangle (11,9);
\path[fill=gray,opacity=.25] (11,9) rectangle (13,7);
\path[fill=gray,opacity=.25] (13,7) rectangle (20,0);
\foreach \x in {1,...,19}
{\draw (\x,0) node[below]{$\x$} to (\x,20);}
\foreach \y in {1,...,19}
{\draw (0,\y) node[left]{$\y$} to (20,\y);}
\draw[thick,->] (0,0) to (20,0) node[right]{$t$};
\draw[thick,->] (0,0) to (0,20) node[above]{$b$};
\draw[very thick, myorange] (3,2) to (20,2);
\draw[very thick, myorange] (5,9) to (11,9);
\draw[very thick, myorange] (3,2) to (3,20);
\draw[very thick, myorange] (6,15) to (9,15);
\draw[very thick, myorange] (5,9) to (5,20);
\draw[very thick, myorange] (13,4) to (20,4);
\draw[very thick, myorange] (6,15) to (6,20);
\draw[very thick, myorange] (11,7) to (13,7);
\draw[very thick, myorange] (9,13) to (14,13);
\draw[very thick, myorange] (12,18) to (17,18);
\draw[very thick, myorange] (9,13) to (9,15);
\draw[very thick, myorange] (14,11) to (15,11);
\draw[very thick, myorange] (16,19) to (18,19);
\draw[very thick, myorange] (11,7) to (11,9);
\draw[very thick, myorange] (12,18) to (12,20);
\draw[very thick, myorange] (17,14) to (20,14);
\draw[very thick, myorange] (13,4) to (13,7);
\draw[very thick, myorange] (14,11) to (14,13);
\draw[very thick, myorange] (15,10) to (20,10);
\draw[very thick, myorange] (15,10) to (15,11);
\draw[very thick, myorange] (16,19) to (16,20);
\draw[very thick, myorange] (18,17) to (19,17);
\draw[very thick, myorange] (17,14) to (17,18);
\draw[very thick, myorange] (18,17) to (18,19);
\draw[very thick, myorange] (19,16) to (20,16);
\draw[very thick, myorange] (19,16) to (19,17);
\node at (6,15) {$\bullet$}; \node at (9,13) {$\bullet$}; \node at (11,7) {$\bullet$}; \node at (13,4) {$\bullet$};
\end{tikzpicture}
\]
Continuing with this example, 
the next step in the proof of Proposition \ref{prop:corners} identifies the
outer {\color{myorange} \textbf{orange}} corners with inner {\color{red} \textbf{red}} corners.
Lemma \ref{lem:algorithm} then passes from
these inner {\color{red} \textbf{red}} corners to 
the indicated sequence of outer {\color{red} \textbf{red}} corners
(here we display the {\color{red} \textbf{red}} and {\color{myorange} \textbf{orange}} 
shadow diagrams):
\[
\begin{tikzpicture}[anchorbase, scale=.3,tinynodes]
\path[fill=gray,opacity=.25] (0,20) rectangle (6,15);
\path[fill=gray,opacity=.25] (6,15) rectangle (9,13);
\path[fill=gray,opacity=.25] (9,13) rectangle (11,7);
\path[fill=gray,opacity=.25] (11,7) rectangle (13,4);
\path[fill=gray,opacity=.25] (13,4) rectangle (20,0);
\foreach \x in {1,...,19}
{\draw (\x,0) node[below]{$\x$} to (\x,20);}
\foreach \y in {1,...,19}
{\draw (0,\y) node[left]{$\y$} to (20,\y);}
\draw[thick,->] (0,0) to (20,0) node[right]{$t$};
\draw[thick,->] (0,0) to (0,20) node[above]{$b$};
\draw[very thick, red] (1,2) to (1,20);
\draw[very thick, red] (1,2) to (3,2); \draw[very thick, myorange] (3,2) to (20,2);
\draw[very thick, red] (2,9) to (2,20);
\draw[very thick, red] (2,9) to (5,9); \draw[very thick, myorange] (5,9) to (11,9);
\draw[very thick, red] (3,1) to (3,2); \draw[very thick, myorange] (3,2) to (3,20);
\draw[very thick, red] (3,1) to (20,1);
\draw[very thick, red] (4,15) to (4,20);
\draw[very thick, red] (4,15) to (6,15); \draw[very thick, myorange] (6,15) to (9,15);
\draw[very thick, red] (5,4) to (5,9); \draw[very thick, myorange] (5,9) to (5,20);
\draw[very thick, red] (5,4) to (13,4); \draw[very thick, myorange] (13,4) to (20,4);
\draw[very thick, red] (6,7) to (6,15); \draw[very thick, myorange] (6,15) to (6,20);
\draw[very thick, red] (6,7) to (11,7); \draw[very thick, myorange] (11,7) to (13,7);
\draw[very thick, red] (7,13) to (7,20);
\draw[very thick, red] (7,13) to (9,13); \draw[very thick, myorange] (9,13) to (14,13);
\draw[very thick, red] (8,18) to (8,20);
\draw[very thick, red] (8,18) to (12,18); \draw[very thick, myorange] (12,18) to (17,18);
\draw[very thick, red] (9,11) to (9,13); \draw[very thick, myorange] (9,13) to (9,15);
\draw[very thick, red] (9,11) to (14,11); \draw[very thick, myorange] (14,11) to (15,11);
\draw[very thick, red] (10,19) to (10,20);
\draw[very thick, red] (10,19) to (16,19); \draw[very thick, myorange] (16,19) to (18,19);
\draw[very thick, red] (11,5) to (11,7); \draw[very thick, myorange] (11,7) to (11,9);
\draw[very thick, red] (11,5) to (20,5);
\draw[very thick, red] (12,14) to (12,18); \draw[very thick, myorange] (12,18) to (12,20);
\draw[very thick, red] (12,14) to (17,14); \draw[very thick, myorange] (17,14) to (20,14);
\draw[very thick, red] (13,3) to (13,4); \draw[very thick, myorange] (13,4) to (13,7);
\draw[very thick, red] (13,3) to (20,3);
\draw[very thick, red] (14,10) to (14,11); \draw[very thick, myorange] (14,11) to (14,13);
\draw[very thick, red] (14,10) to (15,10); \draw[very thick, myorange] (15,10) to (20,10);
\draw[very thick, red] (15,6) to (15,10); \draw[very thick, myorange] (15,10) to (15,11);
\draw[very thick, red] (15,6) to (20,6);
\draw[very thick, red] (16,17) to (16,19); \draw[very thick, myorange] (16,19) to (16,20);
\draw[very thick, red] (16,17) to (18,17); \draw[very thick, myorange] (18,17) to (19,17);
\draw[very thick, red] (17,8) to (17,14); \draw[very thick, myorange] (17,14) to (17,18);
\draw[very thick, red] (17,8) to (20,8);
\draw[very thick, red] (18,16) to (18,17); \draw[very thick, myorange] (18,17) to (18,19);
\draw[very thick, red] (18,16) to (19,16); \draw[very thick, myorange] (19,16) to (20,16);
\draw[very thick, red] (19,12) to (19,16); \draw[very thick, myorange] (19,16) to (19,17);
\draw[very thick, red] (19,12) to (20,12);
\node at (6,15) {$\bullet$}; \node at (9,13) {$\bullet$}; \node at (11,7) {$\bullet$}; \node at (13,4) {$\bullet$};
\end{tikzpicture}
\implies
\begin{tikzpicture}[anchorbase, scale=.3,tinynodes]
\path[fill=gray,opacity=.25] (0,20) rectangle (6,15);
\path[fill=gray,opacity=.25] (6,15) rectangle (9,13);
\path[fill=gray,opacity=.25] (9,13) rectangle (11,7);
\path[fill=gray,opacity=.25] (11,7) rectangle (13,4);
\path[fill=gray,opacity=.25] (13,4) rectangle (20,0);
\foreach \x in {1,...,19}
{\draw (\x,0) node[below]{$\x$} to (\x,20);}
\foreach \y in {1,...,19}
{\draw (0,\y) node[left]{$\y$} to (20,\y);}
\draw[thick,->] (0,0) to (20,0) node[right]{$t$};
\draw[thick,->] (0,0) to (0,20) node[above]{$b$};
\draw[very thick, red] (1,2) to (1,20);
\draw[very thick, red] (1,2) to (3,2); \draw[very thick, myorange] (3,2) to (20,2);
\draw[very thick, red] (2,9) to (2,20);
\draw[very thick, red] (2,9) to (5,9); \draw[very thick, myorange] (5,9) to (11,9);
\draw[very thick, red] (3,1) to (3,2); \draw[very thick, myorange] (3,2) to (3,20);
\draw[very thick, red] (3,1) to (20,1);
\draw[very thick, red] (4,15) to (4,20);
\draw[very thick, red] (4,15) to (6,15); \draw[very thick, myorange] (6,15) to (9,15);
\draw[very thick, red] (5,4) to (5,9); \draw[very thick, myorange] (5,9) to (5,20);
\draw[very thick, red] (5,4) to (13,4); \draw[very thick, myorange] (13,4) to (20,4);
\draw[very thick, red] (6,7) to (6,15); \draw[very thick, myorange] (6,15) to (6,20);
\draw[very thick, red] (6,7) to (11,7); \draw[very thick, myorange] (11,7) to (13,7);
\draw[very thick, red] (7,13) to (7,20);
\draw[very thick, red] (7,13) to (9,13); \draw[very thick, myorange] (9,13) to (14,13);
\draw[very thick, red] (8,18) to (8,20);
\draw[very thick, red] (8,18) to (12,18); \draw[very thick, myorange] (12,18) to (17,18);
\draw[very thick, red] (9,11) to (9,13); \draw[very thick, myorange] (9,13) to (9,15);
\draw[very thick, red] (9,11) to (14,11); \draw[very thick, myorange] (14,11) to (15,11);
\draw[very thick, red] (10,19) to (10,20);
\draw[very thick, red] (10,19) to (16,19); \draw[very thick, myorange] (16,19) to (18,19);
\draw[very thick, red] (11,5) to (11,7); \draw[very thick, myorange] (11,7) to (11,9);
\draw[very thick, red] (11,5) to (20,5);
\draw[very thick, red] (12,14) to (12,18); \draw[very thick, myorange] (12,18) to (12,20);
\draw[very thick, red] (12,14) to (17,14); \draw[very thick, myorange] (17,14) to (20,14);
\draw[very thick, red] (13,3) to (13,4); \draw[very thick, myorange] (13,4) to (13,7);
\draw[very thick, red] (13,3) to (20,3);
\draw[very thick, red] (14,10) to (14,11); \draw[very thick, myorange] (14,11) to (14,13);
\draw[very thick, red] (14,10) to (15,10); \draw[very thick, myorange] (15,10) to (20,10);
\draw[very thick, red] (15,6) to (15,10); \draw[very thick, myorange] (15,10) to (15,11);
\draw[very thick, red] (15,6) to (20,6);
\draw[very thick, red] (16,17) to (16,19); \draw[very thick, myorange] (16,19) to (16,20);
\draw[very thick, red] (16,17) to (18,17); \draw[very thick, myorange] (18,17) to (19,17);
\draw[very thick, red] (17,8) to (17,14); \draw[very thick, myorange] (17,14) to (17,18);
\draw[very thick, red] (17,8) to (20,8);
\draw[very thick, red] (18,16) to (18,17); \draw[very thick, myorange] (18,17) to (18,19);
\draw[very thick, red] (18,16) to (19,16); \draw[very thick, myorange] (19,16) to (20,16);
\draw[very thick, red] (19,12) to (19,16); \draw[very thick, myorange] (19,16) to (19,17);
\draw[very thick, red] (19,12) to (20,12);
\node at (4,15) {$\bullet$}; \node at (7,13) {$\bullet$}; \node at (9,11) {$\bullet$}; 
	\node at (11,5) {$\bullet$}; \node at (13,3) {$\bullet$};
\end{tikzpicture}
\]
\end{example}

\smallskip

By pairing Theorem \ref{thm:V} with Proposition \ref{prop:corners}, 
we deduce the decreasing sequence case of Theorem \ref{thm:longest}.

\begin{cor}\label{cor:decreasing}
Let $w = a_1 \cdots a_k$ be a permutation in $\SG_k$ written in one-line notation. 
The number of rows in $\RS(w)$ equals the length of the longest 
decreasing subsequence in the list $a_1 , \ldots , a_k$.
\end{cor}
\begin{proof}
By Theorem \ref{thm:V}, the number of rows in the Young diagram 
associated to $w$ equals the largest color $c_\ell$ appearing in $w$'s Viennot diagram. 
By Proposition \ref{prop:corners}, $\ell$ is the length of the longest 
decreasing subsequence in $a_1 , \ldots , a_k$.
\end{proof}

For completeness, 
we also recall Viennot's proof of the increasing subsequence case of Theorem \ref{thm:longest}.

\begin{prop}\label{prop:increasing}
The number of columns in $\RS(w)$ equals the length of the longest increasing subsequence.
\end{prop}

\begin{proof}
By Theorem \ref{thm:V},
the number of columns is equal to the number of $c_1$-colored shadow lines
in the Viennot diagram for $w$.

We may produce an increasing subsequence with one 
outer corner on each $c_1$-colored shadow line as follows. 
Start with any outer corner on the ``innermost'' $c_1$-colored shadow line 
(i.e. the $c_1$-colored shadow line that intersects the line $b=k+1$ at the largest $t$ value).
Travel down from this point until intersecting the next $c_1$-colored shadow line, 
then to the left until an outer corner is reached. 
Continuing in this way, we identify an outer corner on each $c_1$ colored shadow line,
and this sequence of outer corners is increasing by construction.

Conversely, the entries of any increasing subsequence determine outer corners that 
must lie on distinct $c_1$-colored shadow lines 
(since the shadow lines in a Viennot diagram always travel downward and rightward), 
so the length of any increasing subsequence is bounded above by the number of 
$c_1$-colored shadow lines.
\end{proof}

As presented in Definition \ref{def:Viennot1}, Viennot's construction
has a manifest symmetry, given by reflection across the diagonal $t=b$,
which immediately implies a celebrated and non-obvious property of the Robinson--Schensted correspondence: 
that interchanging the insertion and recording tableaux 
corresponds to taking the inverse of the corresponding permutation. 
However, for our generalization in \S \ref{sec:updown} it will be useful to break this symmetry, 
and view Viennot's construction as a certain timeline that encodes Schensted's bumping algorithm 
via the following alternative description.

\begin{defn}[Viennot's geometric construction, alternate formulation]\label{def:Viennot}
Let $w \in \SG_k$ be a permutation written in one-line notation as $w = a_1 \cdots a_k$, 
and let $\CS_k=\{c_1 < \cdots < c_k\}$ be a totally ordered set of $k$ colors (also let $c_0$ denote no color). 
The \emph{Viennot diagram} of $w$ is the collection of points 
and colored segments in the first quadrant 
$\{(t,b) \mid t,b > 0 \} \subset \R^2$ defined as follows. 
For each ``time coordinate'' $t=1,2,\ldots,k$, repeat the following steps:
\begin{itemize}
\item Draw the point $(t,a_t)$.
\item Draw the vertical segment from $(t,a_t)$ to $(t,k+1)$ and color it as follows: 
it begins with color $c_1$ and changes color from $c_i$ to $c_{i+1}$ whenever it passes a point 
$(t,b)$ where the horizontal segment $(t-1,b) \to (t,b)$ is colored $c_i$.
\item Draw horizontal segments $(t,a_i) \to (t+1,a_i)$ for all $1 \leq i \leq t$ colored as follows:
if the horizontal segment $(t-1,a_i) \to (t,a_i)$ and the vertical segment $(t,a_i-1) \to (t,a_i)$ have the same color $c_j$, then color 
this segment $c_{j+1}$. Otherwise, color it the same as the horizontal segment $(t-1,a_i) \to (t,a_i)$.
\end{itemize}
The Viennot diagram is the resulting diagram after step $k$.
\end{defn}

It is straightforward (and well-known \cite{ViennotWiki}) 
that this description is equivalent to Definition \ref{def:Viennot1}.

\begin{example}\label{ex:timeline}
We illustrate the steps in Definition \ref{def:Viennot} for the 
permutation $w=2 \ 4 \ 3 \ 1$ in $\SG_4$ from the (running) Example \ref{ex:SaganV}. 
Our colors are 
{\color{red}\textbf{red}} $<$ {\color{myorange}\textbf{orange}} $<$ {\color{green} \textbf{green}} 
($<$ {\color{blue}\textbf{blue}}).
\[
\begin{tikzpicture}[anchorbase, scale=.5,smallnodes]
\foreach \x in {1,...,4}
{\draw (\x,0) node[below]{$\x$} to (\x,5);}
\foreach \y in {1,...,4}
{\draw (0,\y) node[left]{$\y$} to (5,\y);}
\draw[thick,->] (0,0) to (5,0) node[right]{$t$};
\draw[thick,->] (0,0) to (0,5) node[above]{$b$};
\draw[very thick, red] (1,2) to (1,5);
\draw[very thick, red] (1,2) to (2,2);
\node at (1,2){$\bullet$};
\end{tikzpicture}
\, , \quad
\begin{tikzpicture}[anchorbase, scale=.5,smallnodes]
\foreach \x in {1,...,4}
{\draw (\x,0) node[below]{$\x$} to (\x,5);}
\foreach \y in {1,...,4}
{\draw (0,\y) node[left]{$\y$} to (5,\y);}
\draw[thick,->] (0,0) to (5,0) node[right]{$t$};
\draw[thick,->] (0,0) to (0,5) node[above]{$b$};
\draw[very thick, red] (1,2) to (1,5);
\draw[very thick, red] (1,2) to (3,2);
\node at (1,2){$\bullet$};
\draw[very thick, red] (2,4) to (2,5);
\draw[very thick, red] (2,4) to (3,4);
\node at (2,4){$\bullet$};
\end{tikzpicture}
\, , \quad
\begin{tikzpicture}[anchorbase, scale=.5,smallnodes]
\foreach \x in {1,...,4}
{\draw (\x,0) node[below]{$\x$} to (\x,5);}
\foreach \y in {1,...,4}
{\draw (0,\y) node[left]{$\y$} to (5,\y);}
\draw[thick,->] (0,0) to (5,0) node[right]{$t$};
\draw[thick,->] (0,0) to (0,5) node[above]{$b$};
\draw[very thick, red] (1,2) to (1,5);
\draw[very thick, red] (1,2) to (4,2);
\node at (1,2){$\bullet$};
\draw[very thick, red] (2,4) to (2,5);
\draw[very thick, red] (2,4) to (3,4); \draw[very thick, myorange] (3,4) to (4,4); 
\node at (2,4){$\bullet$};
\draw[very thick, red] (3,3) to (3,4); \draw[very thick, myorange] (3,4) to (3,5); 
\draw[very thick, red] (3,3) to (4,3);
\node at (3,3){$\bullet$};
\end{tikzpicture}
\, , \quad
\begin{tikzpicture}[anchorbase, scale=.5,smallnodes]
\foreach \x in {1,...,4}
{\draw (\x,0) node[below]{$\x$} to (\x,5);}
\foreach \y in {1,...,4}
{\draw (0,\y) node[left]{$\y$} to (5,\y);}
\draw[thick,->] (0,0) to (5,0) node[right]{$t$};
\draw[thick,->] (0,0) to (0,5) node[above]{$b$};
\draw[very thick, red] (1,2) to (1,5);
\draw[very thick, red] (1,2) to (4,2); \draw[very thick, myorange] (4,2) to (5,2); 
\node at (1,2){$\bullet$};
\draw[very thick, red] (2,4) to (2,5);
\draw[very thick, red] (2,4) to (3,4); \draw[very thick, myorange] (3,4) to (4,4); 
	\draw[very thick, green] (4,4) to (5,4);  
\node at (2,4){$\bullet$};
\draw[very thick, red] (3,3) to (3,4); \draw[very thick, myorange] (3,4) to (3,5); 
\draw[very thick, red] (3,3) to (5,3);
\node at (3,3){$\bullet$};
\draw[very thick, red] (4,1) to (4,2); \draw[very thick, myorange] (4,2) to (4,4); 
	\draw[very thick, green] (4,4) to (4,5);  
\draw[very thick, red] (4,1) to (5,1);
\node at (4,1){$\bullet$};
\end{tikzpicture}
\]
\end{example}

\begin{rem}\label{rem:timeline}
The proof of Theorem \ref{thm:V} is essentially the 
alternative description of the Viennot diagram in Definition \ref{def:Viennot}, 
together with the observation that the alternative description encodes Schensted's bumping algorithm.
Indeed,
at time $t$ the value $a_t$ enters the insertion tableaux. 
The (horizontal) colored segment $(t,a_t) \to (k+1,a_t)$ tracks the rows to which this value gets bumped
as new values get added. The (vertical) colored segment $(t,a_t) \to (t,k+1)$ encodes the bumping that results
when the value $a_t$ is inserted.

If we stop the algorithm in Definition \ref{def:Viennot} after $s$ steps, 
we obtain a diagram inside a $[1,s] \times [1,k+1]$ box. 
Reading tableaux off the right and top walls of the box as in Theorem \ref{thm:V}, 
we recover the triple $(P_s, Q_s, \lambda_s)$. 
In other words, $P_s$ is determined by the colors meeting the vertical line at $t = s + \frac{1}{2}$, 
while $Q_s$ is determined by the colors meeting the horizontal line at $b = k+1$, as before.
\end{rem}

\begin{example}
Continuing Example \ref{ex:timeline}, we have the following $P_s$ and $Q_s$ for each step $s=1, 2, 3, 4$.
\[
{\footnotesize
\begin{ytableau}
2
\end{ytableau}
}
\, , \ 
{\footnotesize
\begin{ytableau}
1
\end{ytableau}
}
\qquad \qquad
{\footnotesize
\begin{ytableau}
2 & 4
\end{ytableau}
}
\, , \ 
{\footnotesize
\begin{ytableau}
1 & 2
\end{ytableau}
}
\qquad \qquad
{\footnotesize
\begin{ytableau}
2 & 3 \cr
4
\end{ytableau}
}
\, , \ 
{\footnotesize
\begin{ytableau}
1 & 2  \cr
3
\end{ytableau}
}
\qquad \qquad 
{\footnotesize
\begin{ytableau}
1 & 3  \cr
2  \cr
4
\end{ytableau}
}
\, , \ 
{\footnotesize
\begin{ytableau}
1 & 2  \cr
3  \cr
4
\end{ytableau}
}
\]
\end{example}

%
\section{Oscillating Viennot diagrams and patterns}\label{sec:updown}
%

We now extend Viennot's geometric construction to oscillating tableaux. 

\begin{defn}\label{def:UDV}
Let $\mathcal{M}$ be a matching of $2k$ points and let 
$w(\mathcal{M}) = a_1 \cdots a_{2k}$ be the corresponding matching word. 
The \emph{oscillating Viennot diagram} of $\mathcal{M}$ is the collection
of marked points and colored segments in the first quadrant 
$\{(t,b) \mid t,b > 0 \} \subset \R^2$ defined as follows. 
For each ``time coordinate'' $t=1,2,\ldots,2k$, do the following:
\begin{enumerate}
\item If $a_t$ is un-barred, draw the marking $\bullet$ at the point $(t,a_t)$ 
and draw horizontal and vertical segments following the procedure from Definition \ref{def:Viennot}.

\item If $a_t$ is barred, draw the marking $\times$ at the point $(t,a_t)$ and draw horizontal segments 
as in Definition \ref{def:Viennot} at all $b$-values except $b=a_t$. The horizontal segment at height 
$a_t$ ends at the point $(t,a_t)$, and no vertical segment is drawn starting from this point.
\end{enumerate}
The oscillating Viennot diagram, denoted $V(\mathcal{M})$, 
is the resulting diagram after step $2k$.
\end{defn}

\begin{example}\label{ex:UDtimeline}
We illustrate the steps in the above algorithm for the matching word
$4 \ 2 \ 3 \ \bar{4} \ \bar{3} \ 1 \ \bar{2} \ \bar{1}$.
Our colors are 
{\color{red}\textbf{red}} $<$ {\color{myorange}\textbf{orange}} $<$ {\color{green} \textbf{green}} 
($<$ {\color{blue}\textbf{blue}}).
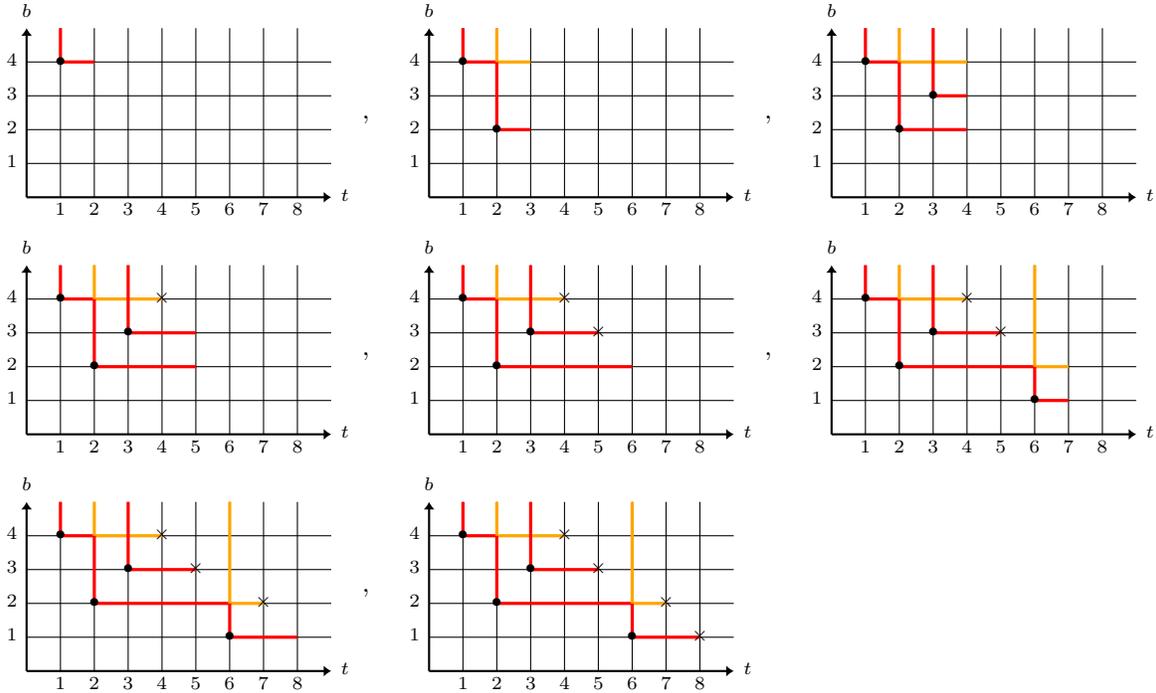
\[
\begin{aligned}
\begin{tikzpicture}[anchorbase, scale=.45,smallnodes]
\foreach \x in {1,...,8}
{\draw (\x,0) node[below]{$\x$} to (\x,5);}
\foreach \y in {1,...,4}
{\draw (0,\y) node[left]{$\y$} to (9,\y);}
\draw[thick,->] (0,0) to (9,0) node[right]{$t$};
\draw[thick,->] (0,0) to (0,5) node[above]{$b$};
\draw[very thick, red] (1,4) to (1,5);
\draw[very thick, red] (1,4) to (2,4);
\node at (1,4){$\bullet$};
\end{tikzpicture}
\, &, \quad
\begin{tikzpicture}[anchorbase, scale=.45,smallnodes]
\foreach \x in {1,...,8}
{\draw (\x,0) node[below]{$\x$} to (\x,5);}
\foreach \y in {1,...,4}
{\draw (0,\y) node[left]{$\y$} to (9,\y);}
\draw[thick,->] (0,0) to (9,0) node[right]{$t$};
\draw[thick,->] (0,0) to (0,5) node[above]{$b$};
\draw[very thick, red] (1,4) to (1,5);
\draw[very thick, red] (1,4) to (2,4); \draw[very thick, myorange] (2,4) to (3,4);
\node at (1,4){$\bullet$};
\draw[very thick, red] (2,2) to (2,4); \draw[very thick, myorange] (2,4) to (2,5);
\draw[very thick, red] (2,2) to (3,2);
\node at (2,2){$\bullet$};
\end{tikzpicture}
\, , \quad
\begin{tikzpicture}[anchorbase, scale=.45,smallnodes]
\foreach \x in {1,...,8}
{\draw (\x,0) node[below]{$\x$} to (\x,5);}
\foreach \y in {1,...,4}
{\draw (0,\y) node[left]{$\y$} to (9,\y);}
\draw[thick,->] (0,0) to (9,0) node[right]{$t$};
\draw[thick,->] (0,0) to (0,5) node[above]{$b$};
\draw[very thick, red] (1,4) to (1,5);
\draw[very thick, red] (1,4) to (2,4); \draw[very thick, myorange] (2,4) to (4,4);
\node at (1,4){$\bullet$};
\draw[very thick, red] (2,2) to (2,4); \draw[very thick, myorange] (2,4) to (2,5);
\draw[very thick, red] (2,2) to (4,2);
\node at (2,2){$\bullet$};
\draw[very thick, red] (3,3) to (3,5);
\draw[very thick, red] (3,3) to (4,3);
\node at (3,3){$\bullet$};
\end{tikzpicture} \\
\begin{tikzpicture}[anchorbase, scale=.45,smallnodes]
\foreach \x in {1,...,8}
{\draw (\x,0) node[below]{$\x$} to (\x,5);}
\foreach \y in {1,...,4}
{\draw (0,\y) node[left]{$\y$} to (9,\y);}
\draw[thick,->] (0,0) to (9,0) node[right]{$t$};
\draw[thick,->] (0,0) to (0,5) node[above]{$b$};
\draw[very thick, red] (1,4) to (1,5);
\draw[very thick, red] (1,4) to (2,4); \draw[very thick, myorange] (2,4) to (4,4);
\node at (1,4){$\bullet$};
\draw[very thick, red] (2,2) to (2,4); \draw[very thick, myorange] (2,4) to (2,5);
\draw[very thick, red] (2,2) to (5,2);
\node at (2,2){$\bullet$};
\draw[very thick, red] (3,3) to (3,5);
\draw[very thick, red] (3,3) to (5,3);
\node at (3,3){$\bullet$};
\node at (4,4){$\times$};
\end{tikzpicture}
\, &, \quad
\begin{tikzpicture}[anchorbase, scale=.45,smallnodes]
\foreach \x in {1,...,8}
{\draw (\x,0) node[below]{$\x$} to (\x,5);}
\foreach \y in {1,...,4}
{\draw (0,\y) node[left]{$\y$} to (9,\y);}
\draw[thick,->] (0,0) to (9,0) node[right]{$t$};
\draw[thick,->] (0,0) to (0,5) node[above]{$b$};
\draw[very thick, red] (1,4) to (1,5);
\draw[very thick, red] (1,4) to (2,4); \draw[very thick, myorange] (2,4) to (4,4);
\node at (1,4){$\bullet$};
\draw[very thick, red] (2,2) to (2,4); \draw[very thick, myorange] (2,4) to (2,5);
\draw[very thick, red] (2,2) to (6,2);
\node at (2,2){$\bullet$};
\draw[very thick, red] (3,3) to (3,5);
\draw[very thick, red] (3,3) to (5,3);
\node at (3,3){$\bullet$};
\node at (4,4){$\times$};
\node at (5,3){$\times$};
\end{tikzpicture}
\, , \quad
\begin{tikzpicture}[anchorbase, scale=.45,smallnodes]
\foreach \x in {1,...,8}
{\draw (\x,0) node[below]{$\x$} to (\x,5);}
\foreach \y in {1,...,4}
{\draw (0,\y) node[left]{$\y$} to (9,\y);}
\draw[thick,->] (0,0) to (9,0) node[right]{$t$};
\draw[thick,->] (0,0) to (0,5) node[above]{$b$};
\draw[very thick, red] (1,4) to (1,5);
\draw[very thick, red] (1,4) to (2,4); \draw[very thick, myorange] (2,4) to (4,4);
\node at (1,4){$\bullet$};
\draw[very thick, red] (2,2) to (2,4); \draw[very thick, myorange] (2,4) to (2,5);
\draw[very thick, red] (2,2) to (6,2); \draw[very thick, myorange] (6,2) to (7,2);
\node at (2,2){$\bullet$};
\draw[very thick, red] (3,3) to (3,5);
\draw[very thick, red] (3,3) to (5,3);
\node at (3,3){$\bullet$};
\node at (4,4){$\times$};
\node at (5,3){$\times$};
\draw[very thick, red] (6,1) to (6,2); \draw[very thick, myorange] (6,2) to (6,5);
\draw[very thick, red] (6,1) to (7,1);
\node at (6,1){$\bullet$};
\end{tikzpicture} \\
\begin{tikzpicture}[anchorbase, scale=.45,smallnodes]
\foreach \x in {1,...,8}
{\draw (\x,0) node[below]{$\x$} to (\x,5);}
\foreach \y in {1,...,4}
{\draw (0,\y) node[left]{$\y$} to (9,\y);}
\draw[thick,->] (0,0) to (9,0) node[right]{$t$};
\draw[thick,->] (0,0) to (0,5) node[above]{$b$};
\draw[very thick, red] (1,4) to (1,5);
\draw[very thick, red] (1,4) to (2,4); \draw[very thick, myorange] (2,4) to (4,4);
\node at (1,4){$\bullet$};
\draw[very thick, red] (2,2) to (2,4); \draw[very thick, myorange] (2,4) to (2,5);
\draw[very thick, red] (2,2) to (6,2); \draw[very thick, myorange] (6,2) to (7,2);
\node at (2,2){$\bullet$};
\draw[very thick, red] (3,3) to (3,5);
\draw[very thick, red] (3,3) to (5,3);
\node at (3,3){$\bullet$};
\node at (4,4){$\times$};
\node at (5,3){$\times$};
\draw[very thick, red] (6,1) to (6,2); \draw[very thick, myorange] (6,2) to (6,5);
\draw[very thick, red] (6,1) to (8,1);
\node at (6,1){$\bullet$};
\node at (7,2){$\times$};
\end{tikzpicture}
\, &, \quad
\begin{tikzpicture}[anchorbase, scale=.45,smallnodes]
\foreach \x in {1,...,8}
{\draw (\x,0) node[below]{$\x$} to (\x,5);}
\foreach \y in {1,...,4}
{\draw (0,\y) node[left]{$\y$} to (9,\y);}
\draw[thick,->] (0,0) to (9,0) node[right]{$t$};
\draw[thick,->] (0,0) to (0,5) node[above]{$b$};
\draw[very thick, red] (1,4) to (1,5);
\draw[very thick, red] (1,4) to (2,4); \draw[very thick, myorange] (2,4) to (4,4);
\node at (1,4){$\bullet$};
\draw[very thick, red] (2,2) to (2,4); \draw[very thick, myorange] (2,4) to (2,5);
\draw[very thick, red] (2,2) to (6,2); \draw[very thick, myorange] (6,2) to (7,2);
\node at (2,2){$\bullet$};
\draw[very thick, red] (3,3) to (3,5);
\draw[very thick, red] (3,3) to (5,3);
\node at (3,3){$\bullet$};
\node at (4,4){$\times$};
\node at (5,3){$\times$};
\draw[very thick, red] (6,1) to (6,2); \draw[very thick, myorange] (6,2) to (6,5);
\draw[very thick, red] (6,1) to (8,1);
\node at (6,1){$\bullet$};
\node at (7,2){$\times$};
\node at (8,1){$\times$};
\end{tikzpicture}
\end{aligned}
\]
\end{example}

We also give the end result of Definition \ref{def:UDV} for the matching from \eqref{eq:matchingex}.

\begin{example}\label{ex:medium}
The following is the oscillating Viennot diagram for the matching word
\[
7 \ 8 \ 6 \ 5 \ \bar{8} \ 3 \ \bar{7} \ 4 \ 1 \ \bar{6} \ 2 \ \bar{5} \ \bar{4} \ \bar{3} \ \bar{2} \ \bar{1} \, .
\]
Our colors are 
{\color{red}\textbf{red}} $<$ {\color{myorange}\textbf{orange}} $<$ {\color{green} \textbf{green}} 
$<$ {\color{blue}\textbf{blue}} ($< \cdots$).
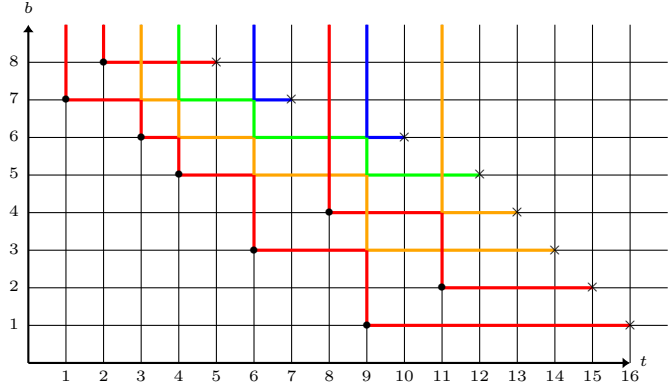
\[
\begin{tikzpicture}[anchorbase, scale=.5,tinynodes]
\foreach \x in {1,...,16}
{\draw (\x,0) node[below]{$\x$} to (\x,9);}
\foreach \y in {1,...,8}
{\draw (0,\y) node[left]{$\y$} to (17,\y);}
\draw[thick,->] (0,0) to (16,0) node[right]{$t$};
\draw[thick,->] (0,0) to (0,9) node[above]{$b$};
\draw[very thick, red] (1,7) to (1,9);
\draw[very thick, red] (1,7) to (3,7); \draw[very thick, myorange] (3,7) to (4,7); 
	\draw[very thick, green] (4,7) to (6,7); \draw[very thick, blue] (6,7) to (7,7);
\node at (1,7){$\bullet$};
\draw[very thick, red] (2,8) to (2,9);
\draw[very thick, red] (2,8) to (5,8); 
\node at (2,8){$\bullet$};
\draw[very thick, red] (3,6) to (3,7); \draw[very thick, myorange] (3,7) to (3,9);
\draw[very thick, red] (3,6) to (4,6); \draw[very thick, myorange] (4,6) to (6,6);
	\draw[very thick, green] (6,6) to (9,6); \draw[very thick, blue] (9,6) to (10,6);
\node at (3,6){$\bullet$};
\draw[very thick, red] (4,5) to (4,6); \draw[very thick, myorange] (4,6) to (4,7);
	\draw[very thick, green] (4,7) to (4,9);
\draw[very thick, red] (4,5) to (6,5); \draw[very thick, myorange] (6,5) to (9,5);
	\draw[very thick, green] (9,5) to (12,5);
\node at (4,5){$\bullet$};
\node at (5,8){$\times$};
\draw[very thick, red] (6,3) to (6,5); \draw[very thick, myorange] (6,5) to (6,6); 
	\draw[very thick, green] (6,6) to (6,7); \draw[very thick, blue] (6,7) to (6,9);
\draw[very thick, red] (6,3) to (9,3); \draw[very thick, myorange] (9,3) to (14,3);
\node at (6,3){$\bullet$};
\node at (7,7){$\times$};
\draw[very thick, red] (8,4) to (8,9);
\draw[very thick, red] (8,4) to (11,4); \draw[very thick, myorange] (11,4) to (13,4);
\node at (8,4){$\bullet$};
\draw[very thick, red] (9,1) to (9,3); \draw[very thick, myorange] (9,3) to (9,5); 
	\draw[very thick, green] (9,5) to (9,6); \draw[very thick, blue] (9,6) to (9,9);
\draw[very thick, red] (9,1) to (16,1); 
\node at (9,1){$\bullet$};
\node at (10,6){$\times$};
\draw[very thick, red] (11,2) to (11,4); \draw[very thick, myorange] (11,4) to (11,9);
\draw[very thick, red] (11,2) to (15,2);
\node at (11,2){$\bullet$};
\node at (12,5){$\times$}; \node at (13,4){$\times$}; \node at (14,3){$\times$}; 
	\node at (15,2){$\times$}; \node at (16,1){$\times$};
\end{tikzpicture}
\]
\end{example}

\begin{prop}\label{prop:obvi}
The assignment $\mathcal{M} \mapsto V(\mathcal{M})$ is a bijection between 
matchings and oscillating Viennot diagrams.
\end{prop}
\begin{proof}
An oscillating Viennot diagram is uniquely determined by the $\bullet$ markings
(the locations of the $c_1$-colored outer corners) 
and the $\times$ markings. These determine, and are determined by, the matching word.
\end{proof}

We now arrive at our generalization of Proposition \ref{prop:corners}.

\begin{prop}\label{prop:main}
Let $\mathcal{M}$ be a matching of $2k$ points and let $c_1 < \cdots < c_k$ 
be the totally ordered set of colors in its oscillating Viennot diagram $V(\mathcal{M})$.
The color $c_\ell$ appears in $V(\mathcal{M})$ if and only if the 
matching word $w(\mathcal{M})$ contains a coexistent decreasing subsequence of length $\ell$. 
\end{prop}

\begin{proof}
If $(t,b)$ is an (outer/inner) corner in an oscillating Viennot diagram, 
then we claim that the rectangle $[1, t] \times [1,b]$ is identical to a corresponding piece in
an ordinary Viennot diagram, i.e. there are no $\times$-marked points inside this rectangle. 
Indeed, the existence of the corner implies that $\bar{b}$ does not appear until after time $t$, 
thus neither does $\bar{b}'$ for any $b' < b$, since barred indices appear in decreasing order. 
Thus all $\times$-marked points in the time interval $[1,t]$ have vertical coordinate in $(b, \infty)$.

Choosing a $c_\ell$-colored outer corner $(t,b)$, 
we can repeat the algorithm of Lemma \ref{lem:algorithm} to obtain a decreasing subsequence of length $\ell$ 
inside the rectangle $[1,t] \times [1,b]$. 
None of the indices in this decreasing subsequence will appear in barred form before time $t$, 
so this sequence is coexistent.

Conversely, suppose we have a length $\ell$ coexistent decreasing subsequence 
$a_{t_1} > a_{t_2} > \cdots > a_{t_\ell}$. 
Then, $\bar{a}_{t_1}$ does not appear until after time $t_{\ell}$, so neither does $\bar{b}$ for any $b < a_{t_1}$. 
Thus the rectangle $[1,t_{\ell}] \times [1, a_{t_1}]$ is identical to a corresponding piece in
an ordinary Viennot diagram. 
The reverse algorithm at the end of Lemma \ref{lem:algorithm} will stay within this rectangle, 
and produces a $c_\ell$-colored outer corner. 
\end{proof}

\begin{proof}[Proof of Theorem \ref{thm:longestC}]
The assignment $V(\mathcal{M}) \mapsto \mathcal{M} \mapsto \UD(\mathcal{M})$
given by composing the bijection from Proposition \ref{prop:obvi} 
with the Sundaram--Stanley bijection is given as in Remark \ref{rem:timeline}, 
i.e. $\UD(\mathcal{M})_s$ can be read off from ``time slices'' in $V(\mathcal{M})$
of the form $t=s+\frac{1}{2}$.
It follows that there exists $0<s<2k$ such that $\UD(\mathcal{M})_s$ has at least $\ell$ rows 
if and only if the color $c_{\ell}$ appears in $V(\mathcal{M})$.
By Proposition \ref{prop:main}, 
the color $c_{\ell}$ appears in $V(\mathcal{M})$ if and only if the matching word $w(\mathcal{M})$ 
has a coexistent decreasing subsequence of length $\ell$, 
which corresponds to an $\ell$-pattern.
\end{proof}

\bibliographystyle{plain}

%

%
\end{document}